\documentclass[smallextended,envcountsect,envcountsame]{svjour3}
\usepackage{graphicx}
\usepackage{mathrsfs}
\usepackage{fix-cm}
\usepackage{amsmath}
\usepackage{amssymb}
\usepackage{latexsym}
\usepackage{dsfont}
\usepackage{xcolor}
\usepackage{cite}
\usepackage{dsfont}
\usepackage{enumitem}

\def\disp{\displaystyle}

\def\Limsup{\mathop{{\rm Lim}\,{\rm sup}}}

\def\tto{\;{\lower 1pt \hbox{$\rightarrow$}}\kern -10pt
\hbox{\raise 2pt \hbox{$\rightarrow$}}\;}

\def\Hat{\widehat}
\def\hat{\widehat}
\def\Tilde{\widetilde}
\def\Bar{\overline}
\def\ra{\rangle}
\def\la{\langle}
\def\ve{\varepsilon}
\def\B{\mathbb{B}}

\def\R{\mathbb{R}}
\def\N{\mathbb{N}}

\def\ox{\bar{x}}
\def\oy{\bar{y}}
\def\ooy{\bar{\y}}

\def\gph{\operatorname{gph}}
\def\epi{\mbox{\rm epi}\,}

\def\dom{\mbox{\rm dom}\,}

\def\cl{\mbox{\rm cl}\,}

\def\inter{\mbox{\rm int}\,}

\def\dn{\downarrow}

\def\tilde{\widetilde}
\def\ph{\varphi}
\def\emp{\emptyset}
\def\st{\stackrel}
\def\oR{\Bar{\R}}
\def\N{\mathbb{N}}

\def\tt{\tilde t}

\def\al{\alpha}

\newcounter{count}

\newcommand{\Intf}[1]{\mathrm{E}_{#1}}
\newcommand{\Intfset}[1]{\mathrm{E}_{#1}}
\newcommand{\IntfLp}[1]{\mathrm{I}_{#1}}

\DeclareMathOperator*{\esssup}{ess\,sup}

\DeclareMathOperator{\1}{\mathds{1}}
\newcommand{\T}{T}
\newcommand{\Rex}{\overline{\mathbb{R}}}
\let\epsilon\varepsilon
\DeclareMathAlphabet{\mathpzc}{OT1}{pzc}{m}{it}

\def\v{\mathpzc{v}}
\def\w{\mathpzc{w}}
\def\x{\mathpzc{x}}
\def\y{\mathpzc{y}}
\def\z{\mathpzc{z}}
\def\X{{\R^n}}
\def\Y{\R^{m}}

\def\Leb{\textnormal{L}}
 \def\Ze{\mathcal{Z}}
 \newcommand{\dmu}{\mu(dt)}
 
\begin{document}

\title{Generalized Sequential Differential Calculus for Expected-Integral Functionals
\thanks{Research of the first author was partially supported by the USA National Science Foundation under grants DMS-1512846 and DMS-1808978, by the USA Air Force Office of Scientific Research under grant \#15RT04, and by the Australian Research Council under Discovery Project DP-190100555. Research of the second author was partially supported by grants: Fondecyt Regular 1190110 and  Fondecyt Regular 1200283.}}

\author{Boris S. Mordukhovich \and \mbox{Pedro P\'erez-Aros}}

\institute{Boris S. Mordukhovich \at Department of Mathematics, Wayne State University, Detroit, Michigan 48202, USA\\ \email{boris@math.wayne.edu} \\
\and Pedro P\'erez-Aros \at Instituto de Ciencias de la Ingenier\'ia, Universidad de O'Higgins, Rancagua, Chile\\
\email{pedro.perez@uoh.com}}

\dedication{Dedicated to Terry Rockafellar, in high esteem}

\date{Received: date / Accepted: date}

\maketitle

\begin{abstract} Motivated by applications to stochastic programming, we introduce and study the expected-integral functionals, which are mappings given in an integral form depending on two variables, the first a finite dimensional decision vector and the second one an integrable function.  The main goal of this paper is to establish sequential versions of Leibniz's rule for regular subgradients by employing and developing appropriate tools of variational analysis.

\keywords{Variational analysis \and generalized differentiation \and stochastic programming \and expected-integral functionals \and sequential calculus}\vspace*{-0.05in}

\subclass{Primary: 49J53, 90C15 \and Secondary: 49J52}\vspace*{-0.2in}
\end{abstract}

\section{Introduction}\label{intro}

Stochastic programming is a branch of optimization which deals with problems under uncertainty with some probabilistic information about the given data. For this class of problems the random phenomena are modeled by using a probability measure space that represents all the possible outcomes, where various classes of integral functionals and set-valued mappings replace random objective functions and constraints; see, e.g., \cite{sdr} for more details and references. For example, given a measure space $(T,\mathcal{A},\mu)$, a random cost function $\ph_t(x)$, and a constraint set $\Omega$, a stochastic program can be formulated as
\begin{align*}
\min\int_T\ph_t(x)\dmu\;\mbox{ subject to }\;x\in\Omega.
\end{align*}
Applying tools of variational analysis together with subdifferential extensions of Leibniz's rule, first-order necessary (and sufficient in some cases) optimality conditions for this problem are formulated in the form
\begin{align*}
0\in\int_T\partial\ph_t(x)\dmu+N(x;\Omega)
\end{align*}
in terms of appropriate subdifferential for nonsmooth functions and normal cones for sets. We refer the reader to \cite{chp19a,chp19b,chp20,mp2020,mor-sag18,mor-sag19} for recent results in this direction and their applications. Note that hereafter the integral of a set-valued mapping is understood in the sense of Aumann \cite{aum}; see below.

Consequently, deriving second-order optimality and stability conditions as well as developing some numerical methods in stochastic programming require the use of generalized differentiation for set-valued mappings. In our approach we relay on {\em coderivatives} of set-valued mappings introduced in Mordukhovich \cite{m80}. While coderivatives have been broadly used in many aspects of deterministic variational analysis, optimization, control theory, etc. (see, e.g., the books \cite{m06,m18,rw} and the references therein), we are not familiar with their applications to stochastic programming. Motivated by such applications, we intend to study coderivatives of set-valued integrals, which unavoidably appear in stochastic programming and related problems. A natural goal in this direction is to obtain a kind of Leibniz's rule for evaluating coderivatives of set-valued integrals via coderivatives of mappings under the integral sign.\vspace*{0.03in}

The present paper is the first part of our study, and we concentrate here on subdifferentiation of appropriate {\em expected-integral functionals}. Such functionals are defined in the form
\begin{equation*}
\Intfset{\varphi}(x,\y):=\int_T\varphi_t(x,\y(t))\dmu,
\end{equation*}
where $x\in\R^n$ and $y\in\Leb^1(T;\R^m)$; see Section~\ref{section3} for the more precise definition and discussion. Note that the minimization of $\Intfset{\varphi}$ is clearly related to {\em two-stage stochastic programming}. Indeed, the integral of the normal integrand applies to a deterministic (first stage) decision variable and a random (second stage) decision variable, where at the second stage the agent has full information and is constrained  to use integrable strategies. However, we have also in mind some other applications of expected-integral functionals; in particular, to problems related to dynamic programming.

To reach our goals, we begin with deriving appropriate versions of Leibniz's rule for regular/Fr\'echet subgradients of $\Intfset{\varphi}$. Since the domain space $\R^n\times \Leb^1(T;\R^m)$ of $\Intfset{\varphi}$ is not Asplund, there exist no results of the required type in either pointwise or fuzzy formats. Achievements of this paper include desired calculus rules in {\em sequential} forms, which are appropriate for our further generalized differentiation theory and applications to stochastic programming. Recall here that the ``sequential" terminology is used in variational analysis to indicate calculus rules and optimality conditions formulated via sequences converging to the reference points versus those stated at the points themselves, which may not be accessible without additional assumptions. To develop the aforementioned sequential calculus, we employ a {\em variational approach} and first obtain sequential necessary optimality conditions for new notions of {\em robust minima}, which are certainly of their own interest. Establishing the latter conditions requires in turn the use and developments of the theories of measurable multifunctions and normal integrals, as well as elaborating techniques of variational analysis and generalized differentiation.\vspace*{0.03in}

The rest of the paper is organized as follow. Section~\ref{section:notation} contains some {\em preliminaries} from variational analysis including the theory of measurable multifunctions and their integrals that are broadly used in the paper. Section~\ref{sec:meas} deals with the notion of {\em graph measurability} and verifies this property for regular subgradient mappings of normal integrands.

In Section~\ref{section3} we introduce and investigate new notions of {\em expected-integral functionals} for normal integrands of two variables and of their {\em $p$-robust minima}. The main result here establishes {\em sequential necessary optimality conditions} for $p$-robust local minimizers in terms of regular subgradients. The last Section~\ref{Section:FuzzyCalculus} is the culmination of the paper, where we obtain two general versions of the {\em sequential Leibniz rule} for expected-integral functionals.\vspace*{-0.2in}

\section{Preliminaries from Variational Analysis}\label{section:notation}

Throughout this paper we use standard notation of variational analysis and generalized differentiation; see, e.g., \cite{m06,rw}. Recall that $\oR:=[-\infty,\infty]$ is the extended real line, and thus a function $\ph\colon\R^n\to\oR$ is extended-real-valued; see, e.g., \cite[page~1]{rw} for the standard rules to deal with infinity. The symbol $\|\cdot\|$ stands to indicate the Euclidean norm on any finite-dimensional space under consideration. Given $x\in\X$ and $r>0$, the closed ball centered at $x$ with radius $r$ is denoted by $\mathbb{B}_r(x)$, while the closed unit ball is simply labeled as $\mathbb{B}$. For a set $\Omega\subset\X$, the symbol $x\overset{\Omega}{\to}\ox$ signifies that $x\to\ox$ with $x\in\Omega$. The {\em indicator function} $\delta_\Omega\colon\X\to\Rex$ of $\Omega$ is defined as $\delta_\Omega(x):=0$ for $x\in\Omega$ and $\delta_\Omega(x):=\infty$ otherwise.\vspace*{0.03in}

Consider a set-valued mapping/multifunction $F\colon\R^n\tto\R^m$ with the {\em domain} and {\em graph} given by
\begin{equation*}
\dom F:=\big\{x\in\R^n\;\big|\;F(x)\ne\emp\} \;\mbox{ and }\;\gph F:=\big\{(x,y)\in\R^n\times\R^m\;\big|\;y\in F(x)\big\},
\end{equation*}
respectively. The {\em Painlev\'e-Kuratowski outer limit} of $F$ as $x\to\ox$ is defined by
\begin{equation}\label{pk}
\Limsup_{x\to\ox}F(x):=\big\{v\in\R^m\big|\;\exists\,\mbox{ seqs. }\;x_k\to\ox,\;v_k\to v\;\mbox{ s.t. }\;v_k\in F(x_k)\big\}.
\end{equation}

Let $\Omega\subset\R^n$ be a set with $\ox\in\Omega$. The {\em regular/Fr\'echet normal cone} to $\Omega$ at $\ox$ is given via the standard upper limit ``$\limsup$" by
\begin{equation}\label{rnc}
\Hat N(\ox;\Omega):=\Big\{x^*\in\R^n\;\Big|\;\limsup_{x\st{\Omega}{\to}\ox}\frac{\la x^*,x-\ox \ra}{\|x-\ox\|}\le 0\Big\}
\end{equation}
with $\Hat N(\ox;\Omega):=\emp$ if $\ox\notin\Omega$. The {\em limiting/Mordukhovich normal cone} to $\Omega$ at $\ox\in\Omega$ is defined via \eqref{pk} by
\begin{equation}\label{lnc}
N(\ox;\Omega):=\Limsup_{x\to\ox}\Hat N(x;\Omega)
\end{equation}
with $N(\ox;\Omega):=\emp$ if $\ox\notin\Omega$. We refer the reader to the books \cite{m06,m18,rw} for these and related subdifferential constructions defined below.

Considering next an extended-real-valued function $\ph\colon\R^n\to\oR$, we associate with it the {\em domain}, i.e., the set $\dom \varphi :=\{  x \in \mathbb{R}^n : f(x) < +\infty  \}$, and the {\em epigraph}
\begin{equation*}
\epi\ph:=\big\{(x,\al)\in\R^{n+1}\;\big|\;\al\ge\ph(x)\big\}.
\end{equation*}
It is said that $\ph$ is {\em proper} if $\dom\ph\ne\emp$. Given $\ox\in\dom\ph$ and based on the normal cones \eqref{rnc} and \eqref{lnc} to the epigraph of $\ph$ at $(\ox,\ph(\ox))$, we define the {\em regular subdifferential} and {\em limiting subdifferential} of $\ph$ at $\ox$ by
\begin{equation}\label{rsub}
\Hat\partial\ph(\ox):=\big\{x^*\in\R^n\;\big|\;(x^*,-1)\in\Hat N\big((\ox,\ph(\ox));\epi\ph\big)\big\},
\end{equation}
\begin{equation}\label{lsub}
\partial\ph(\ox):=\big\{x^*\in\R^n\;\big|\;(x^*,-1)\in N\big((\ox,\ph(\ox));\epi\ph\big)\big\},
\end{equation}
respectively. In the books \cite{m06,m18,rw} and the references therein the reader can find equivalent analytic representations of the subdifferentials \eqref{rsub} and \eqref{lsub}, available calculus rules, and various applications.\vspace*{0.05in}

Next we proceed, following the book \cite{rw}, with recalling the required definitions and preliminary facts from the theory of measurable multifunctions and normal integrands. Throughout the paper,  $(T,\mathcal{A},\mu)$ is a complete finite measure space. As in \cite{rw}, the integral of a measurable extended-real-valued function $\alpha\colon T\to\Rex$ is defined by
\begin{align}\label{def:integral}
\int\limits_T\alpha(t)\dmu:=\int_T\max\big\{\alpha(t),0\big\}\dmu+\int_T\min\big\{\alpha(t),0\big\}\dmu,
\end{align}
with the convention that $\infty +(-\infty)=\infty$.

The {\em characteristic function} of a set $A\in\mathcal{A}$ is given by
\begin{align*}
\1_A (t):=\left\{\begin{array}{ccl}
&1&\text{ if }\;t\in A,\\
&0&\text{ if }\;t\in T\backslash A.
\end{array}\right.
\end{align*}
To avoid confusions, we use in what follows the special font as $\v,\w,\x,\y,\z$, etc. to signify vector-valued measurable functions defined on $T$.

For any $p\in[1,\infty]$ we denote as usual by $\Leb^p({\T},\X)$, with the norm $\|\cdot\|_p$, the set of all (equivalence classes by the relation equal almost everywhere) measurable functions $\x\colon T\to\X$ such that $\|\x\|^p$ is integrable for $p\in[1,\infty)$, and the set of essentially bounded measurable functions for $p=\infty$. Points in $\X$ are identified with constant functions in $\Leb^p(T,\X)$, and thus for a point $x\in\X$ and a function $\x\in\Leb^p(T,\X)$ we use the expressions
\begin{align*}
\|x-\x\|_p&=\left(\int_T\|x-\x(t)\|^p\dmu\right)^{1/p}\text{ for }\;p\in[1,\infty),\\
\|x-\x\|_\infty&=\esssup\limits_{t\in T}\|x-\x(t)\|.
\end{align*}

Considering further a set-valued mapping $M\colon T\tto\R^n$, we say that $M$ is {\em measurable} if $M^{-1}(U)\in\mathcal{A}$ for every open set $U\subset\R^n$, where
$M^{-1}(U):=\{t\in T\;|\;M(t)\cap U\ne\emp\}$. The {\em Aumann integral} of a set-valued mapping $M\colon T\tto\R^n$ over a measurable set $A\in\mathcal{A}$ is defined by
\begin{equation}\label{aum}
\int_A M(t)\dmu:=\Big\{\int_A\x^*(t)\dmu\;\Big|\;\x^*\in {\Leb}^1(\T,\X)\textnormal{ and }\x^*(t)\in M(t)\text{ a.e.}\Big\}.
\end{equation}

Next we recall that $\ph\colon T\times\R^n\to\Rex$ is a {\em normal integrand} if the multifunction $t\to\epi\ph_t$ is measurable with closed values. By the completeness of the measure space, this can be equivalently described as follows: $\ph$ is $\mathcal{A}\otimes\mathcal{B}(\mathbb{R}^n)$-measurable and for every $t\in T$ the function $\ph_t:=\ph(t,\cdot)$ is lower semicontinuous (l.s.c.) on $\R^n$; see, e.g., \cite[Corollary~14.34]{rw}. In addition, we say that the function $\ph$ is proper if $\ph_t$ is proper for all $t\in T$. If $\ph_t$ is a convex for all $t\in T$, we say that $\ph$ is a {\em convex normal integrand}.

Let us now formulate two known integration results dealing with normal integrands. The first proposition verifies the possibility of interchanging the infimum and the integral signs; see, e.g., \cite[Theorem~14.60]{rw}.

\begin{proposition}[interchanging between minimization and integration]\label{dualityRock} Given a normal integrand $\ph\colon T\times\X\to\oR$ and $p\in [1,+\infty)$.  Then we have the equality
\begin{equation*}
\inf_{\x\in\textnormal{L}^p({\T},\X)}\int_T\ph_t\big(\x(t)\big)\dmu=\int_T\inf_{x\in\X}\ph_t(x)\dmu.
\end{equation*}
provided that  $\int_T\ph_t(\x_0(t))\dmu<+\infty$ for at least one $\x_0 \in \Leb^p(T,\X)$.
\end{proposition}

The second preliminary result, taken from \cite[Lemma~37]{gp}, concerns the convergence under the integral sign.

\begin{proposition}[\bf convergence under integral sign]\label{lemma:convergenciaL1} Consider a normal integrand $\ph\colon T\times\X\to\oR$ such that the function $t\to\inf_{x\in\X}\ph_t(x)$ is integrable on $T$. Given a sequence $\{\x_k\}\subset\textnormal{L}^p(T,\X)$ with $\x_k\overset{\textnormal{L}^p}{\to}\x$ as $k\to\infty$ and
\begin{equation*}
\lim_{k\to\infty}\int_T\ph\big(t,\x_k(t)\big)\dmu=\int_T\ph\big(t,\x(t)\big)\dmu\in\R,
\end{equation*}
we have that $\disp\lim_{k\to\infty}\int_T\big|\ph\big(t,\x_k(t)\big)-\ph\big(t,\x(t)\big)\big|\dmu=0$.
\end{proposition}
\vspace*{-0.25in}

\section{Graph Measurability of Subgradient Mappings}\label{sec:meas}\vspace*{-0.05in}

This section deals with the notion of {\em graph measurability} of set-valued mappings, where a particular attention is paid to regular subdifferential graphs of normal integrands that are of special interest in the paper. The graph measurability is understood in the following sense.

\begin{definition}[\bf graph measurability]\label{graph-mes} {\em A set-valued mapping $M\colon T\tto\R^n$  is said to be {\sc graph measurable} if $\gph M\in\mathcal{A}\otimes\mathcal{B}(\mathbb{R}^n)$, where ${\cal B}(\R^n)$ is the Borel $\sigma$-algebra, i.e., the $\sigma$-algebra generated by all open sets of $\R^n$.}
\end{definition}

Since the measure space $(T,\mathcal{A},\mu)$ is assumed to be complete, we can easily observe that a multifunction $M$ with {\em closed values} is graph measurable if and only if it is measurable in the standard sense of Section~\ref{section:notation}. Due to this remark and the closed-graph property of the limiting normal cone \eqref{lnc}, it follows from \cite[Theorems~14.26 and 14.60]{rw} that the set-valued mapping
\begin{equation*}
t\to\gph\partial\ph_t:=\big\{(x,x^\ast)\in\R^{2n}\;\big|\;x^\ast\in\partial\ph_t(x)\big\}
\end{equation*}
generated by the limiting subdifferential \eqref{lsub} of the normal integrand $\ph\colon T\times\R^n\to\oR$ is graph measurable. However, such a device does not work in the case of the regular subgradient mapping \eqref{rsub} for which the values $\gph\Hat\partial\ph_t$ is rarely closed. Nevertheless, the next theorem establishes the desired graph measurability of the regular subgradient mapping that plays a significant role in our subsequent analysis and applications.

\begin{theorem}[\bf graphical measurability of regular subgradient mappings]\label{lemma_measurability_reg_sub} Let $\ph\colon T\times\X\to\Rex$ be a normal integrand. Then the multifunction
\begin{equation*}
t\to\gph\Hat\partial\ph_t:=\big\{(x,x^\ast)\in\R^{2n}\;\big|\;x^\ast\in\Hat\partial\ph_t(x)\big\}
\end{equation*}
is graph measurable on $T$.
\end{theorem}
\begin{proof} Let us split the proof into the following three claims.\\[1ex]
{\bf Claim~1:} {\em We can always assume that $\ph$ is a proper normal integrand}. Indeed, consider the two multifunctions
\begin{equation*}
t\to D(t):=\dom\ph_t\;\mbox{ and }\;t\to L(t):=\big\{x\in\X\;\big|\;\ph_t(x)=-\infty\big\},
\end{equation*}
which both are measurable on $T$ due to \cite[Propositions~14.28 and 14.33]{rw}, respectively. Thus the set
\begin{equation*}
\Tilde T:=\big\{t\in T\;\big|\;\ph_t\;\text{ is proper}\big\}=\dom D\backslash\dom L
\end{equation*}
is a measurable subset of $T$ by the measurability of the set valued functions $D$ and $L$. Furthermore, denoting $S(t):=\gph\Hat\partial\ph_t$ and using \cite[Corollary~2.29]{m06}, we have that $S(t)\ne\emp$ if and only if $\ph_t$ is proper. Since the properness of $\ph$ means the properness of $\ph_t$ for each $t$ from the measurable set under consideration, this is verified for all $t\in\Tilde T$. Thus it is sufficient to prove the measurability of $S$ over $\tilde{T}$, and we suppose in what follows that $\Tilde T=T$ without loss of generality.\\[1ex]
{\bf Claim~2:} {\em Given positive numbers $\epsilon$ and $\gamma$, define the function
\begin{align*}
g_{\epsilon,\gamma}(t,x,x^\ast):=\left\{\begin{array}{cl}
\inf\left\{\Delta_\epsilon(t,w,x,x^\ast)\;\big|\;w\in\inter\mathbb{B}_\gamma(x)\right\}&\text{ if }\;\ph(t,x)<\infty,\\
-\infty&\text{ if }\;\ph(t,x)=\infty,
\end{array}\right.
\end{align*}
where $\Delta_\epsilon(t,w,x,x^\ast):=\ph(t,w)-\ph(t,x)-\langle x^\ast,w-x\rangle+\epsilon\|w-x\|$. Then we have that $g_{\epsilon,\gamma}$ is $\mathcal{A}\otimes\mathcal{B}(\X)\otimes\mathcal{B}(\X)$}-measurable. To verify this statement, let $\{(\x_k,\alpha_k)\}_{k\in\N}$  be a Castaing representation of $\epi\ph_t$ (see, e.g., \cite[Theorem~14.5]{rw}), i.e., $\{(\x_k,\alpha_k)\}_{k\in\N}$ is a sequence of measurable functions such that
\begin{align*}
 \cl\big\{\big(\x_k(t),\alpha_k(t)\big)\;\big|\;k\in\N\big\}=\epi\ph_t\;\text{ for all }\;t\in T.
\end{align*}
Then for each $k\in\N$ define the family of the extended-real-valued functions
\begin{align*}
g^k_{\epsilon,\gamma}(t,x,x^\ast):=\hspace{-0.1cm}\alpha_k(t)-\ph_t(x)-\hspace{-0.1cm}\langle x^\ast,\x_k(t)-x\rangle+\epsilon\|\x_k(t)-x\|\hspace{-0.1cm}+\delta_{{\rm int}\,\mathbb{B}_\gamma (x)}\big(\x_k(t)\big),
\end{align*}
which are measurable on $T$ as sums of measurable functions. We claim that
\begin{equation*}
g_{\epsilon,\gamma}(t,x,x^\ast)=\inf_{k\in\N}g^k_{\epsilon,\gamma}(t,x,x^\ast).
\end{equation*}
Indeed, it follows from $(\x_k(t),\alpha_k(t))\in\epi\ph_t$ that $g^k_{\epsilon,\gamma} (t,x,x^\ast)\ge g_{\epsilon,\gamma}(t,x,x^\ast)$. Picking now any point $w\in\dom\ph_t$ with $w\in\inter \mathbb{B}_\gamma(x)$ (if such a point does not exist, the claimed equality holds trivially) and $\eta>0$, find $k\in\N$ such that
\begin{equation*}
\ph(t,w)\ge\alpha_k(t)-\eta,\;|\langle x^\ast,w-\x_k(t)\rangle|\le\eta,\;\mbox{ and }\;\epsilon\|w-\x_k(t)\|\le\eta.
\end{equation*}
Consequently, we arrive at the estimates
\begin{align*}
\ph(t,w)-\ph(t,x)-\langle x^\ast,w-x\rangle+\epsilon\|w-x\|&\ge g^k_{\epsilon,\gamma}(t,x,x^\ast)-3\eta\\&\ge\inf_{k\in\N}g^k_{\epsilon,\gamma}(t,x,x^\ast)-3\eta.
\end{align*}
Since $w$ and $\eta$ were chosen arbitrarily, it ensures that
\begin{equation*}
g_{\epsilon,\gamma}(t,x,x^\ast)\ge\inf_{k\in\N}g^k_{\epsilon,\gamma}(t,x,x^\ast)
\end{equation*}
and thus verifies this claimed statement.\\[1ex]
{\bf Claim~3:} {\em The multifunction $t\to\gph\Hat\partial\ph_t$ is graph measurable.} Remembering the notation $S(t):=\gph\Hat\partial\ph_t$, we have the representation
\begin{align*}
\gph S=\bigcap\limits_{\epsilon\in(0,1)\cap\mathbb{Q}}\bigcup\limits_{\gamma\in(0,1)\cap\mathbb{Q}}\big\{(t,x,x^\ast)\in T\times\X\times\X\;\big|\;g_{\epsilon,\gamma}(t,x,x^\ast)\ge 0\big\}.
\end{align*}
Taking into account the measurability of the functions $g_{\epsilon,\gamma}(t,x,x^\ast)$ established in Claim~2, the latter representation yields the graph measurability of the mapping $t\to\gph\Hat\partial\ph_t$ and thus completes the proof of the theorem.
\end{proof}

Finally in this section, we present a useful result on measurable selections of graph measurable multifunctions, which may not have closed values. This makes it applicable to the regular subgradient mappings due to Theorem~\ref{lemma_measurability_reg_sub}. The result below and its proof can be found in, e.g., \cite[Theorem III.22]{cas-val}.

\begin{proposition}[\bf measurable selections of graph measurable multifunctions]\label{Prop_measurableselection}
Let $M\colon T\tto\X$ be a graph measurable multifunction with nonempty values $($with nonempty values a.e., respectively$)$. Then there exists a measurable function $\x\colon T\to\X$ such that we have the inclusion $\x(t)\in M(t)$ for all $t\in T$ $($a.e., respectively$)$.
\end{proposition}\vspace*{-0.2in}

\section{Robust Minima of Expected-Integral Functionals}\label{section3}\vspace*{-0.05in}

In this section we define the notions of expected-integral functionals and their robust minima that are crucial for our further considerations. The main result here establishes {\em sequential} necessary conditions for robust minima of expected-integral functionals via regular subgradients of their integrands.\vspace*{0.03in}

Given a normal integrand $\varphi\colon\T\times\X\times\Y\to\Rex$, define the {\em expected-integral functional} $\Intf{\varphi}\colon\X\times\Leb^1(\T,\Y)\to\Rex$ by
\begin{align}\label{eif}
\Intf{\varphi}(x,\y):=\int_{\T}\varphi_t\big(x,\y(t)\big)\dmu,
\end{align}
where the integral is understood in the sense of \eqref{def:integral}. The name of \eqref{eif} is due to the fact that the first variable of $\Intf{\varphi}(x,\y)$ is a point in $\X$ as in the case of expected functionals, while the second variable is an integrable function as in the case of integral functionals. Recall that
\begin{equation*}
\dom \Intf{\varphi}:=\big\{(x,\y)\in\X\times\Leb^1(T,\Y)\;\big|\;\Intf{\varphi}(x,\y)<\infty\big\}.
\end{equation*}
Also, it is convenient to consider in what follows the integral functional $\IntfLp:\colon\Leb^1(\T,\X)\times\Leb^1(\T,\Y)\to\Rex $ defined by
\begin{equation}\label{eifLp}
\IntfLp{\varphi}(\x,\y):=\int_{\T}\varphi_t\big(\x(t),\y(t)\big)\dmu.
\end{equation}

The next definition is fundamental for our study. It presents a certain adaptation of the notions recently introduced in \cite{chp20}.

\begin{definition}[\bf stabilized infimum and robust minimizers]\label{ROBUSTEDLOCALMINIMUM} {\em Let $\Intf{\varphi}$ be the expected-integral functional \eqref{eif} associated with a normal integrand $\varphi\colon\T\times\X\times\Y\to\Rex$, and let $p\in[1,\infty)$.

{\bf(i)} The {\sc $p$-stabilized infimum} of $\Intf{\varphi}$ on the product set $B\times C\subset\X\times\Leb^{1}(T,\Y)$ is defined by
\begin{equation*}
\wedge_{p,B\times C}\Intf{\varphi}:=\sup\limits_{\varepsilon>0}\inf\left\{\displaystyle\int_T\varphi_t\big(\x(t),\y(t)\big)\dmu\,\,\left|\begin{array}{c}
x\in B,\,\y\in C,\,\x\in\Leb^p(T,\X),\\
\displaystyle\int_T\|\x(t)-x\|^p\dmu\le\varepsilon
\end{array}\right.\right\}.
\end{equation*}

{\bf(ii)} The infimum of $\Intf{\varphi}$ on $B\times C$ is called {\sc $p$-robust} if we have
\begin{equation*}
\wedge_{p,B\times C}\Intf{\varphi}=\inf_{B\times C}\Intf{\varphi}\in\R.
\end{equation*}
In that case we say that every minimizer of $\Intf{\varphi}$ on the set $B\times C$ is a {\sc $p$-robust minimizer} of $\Intf{\varphi}$ {\sc on} $B\times C$.

{\bf(iii)} A pair $(x,\y)\in\dom\Intfset{\varphi}$ is called a {\sc $p$-robust local minimizer} of $\Intf{\varphi}$ if it is a $p$-robust minimizer of $\Intf{\varphi}$ on some ball $\mathbb{B}_r(x)\times\mathbb{B}_r(\y)$.}
\end{definition}

From now on we impose the following {\em lower growth condition} on the normal integrand in question: there exist functions $\nu\in\Leb^1(T,\R_+)$ such that
\begin{align}\label{lower_bound_assump}
\varphi_t(v,w)\ge -\nu(t)\;\mbox{ for all }\;v\in\X\;\mbox{ and }\;w\in\Y.
\end{align}
It is not hard to check by using Fatou's lemma that condition \eqref{lower_bound_assump} ensures that the expected-integral functional \eqref{eif} is l.s.c.\ on $\X\times\Leb^1(T,\Y)$; see, e.g., \cite[Lemma~10]{gp} for more details.\vspace*{0.05in}

Before deriving the main result of this section, we present two lemmas. The first one provides well-known results about classical differentiation of integral functionals with normal integrands; see, e.g., \cite{rw} and the references therein.

\begin{lemma}[Leibniz's rules of Fr\'echet differentiation]\label{lemmadifferenciabilidadnorma} Let $\ph\colon T\times\R^n\to\R$ be a normal integrand, which is Lipschitz on an open neighborhood $U$ of $\ox$ with an integrable modulus, i.e., there exists $K\in\Leb^1(T,\R)$ such that
\begin{align*}
|\ph(t,x)-\ph(t,u)|\le K(t)\|x-u\|\;\text{ for all }\;x,u\in U\;\text{ and }\;t\in T.
\end{align*}
If $\ph_t(\cdot)$ is Fr\'echet differentiable at $\ox$ for a.e.\ $t\in T$, then the functional
\begin{equation*}
\Intf{\ph}(x):=\int_T\ph_t(x)\dmu
\end{equation*}
is Fr\'echet differentiable at $\ox$ with $\nabla\ph(\cdot,\ox)\in\textnormal{L}^1(T,\X)$, and we have
\begin{align*}
\nabla\Intf{\ph}(\ox)=\int_{T}\nabla\ph_t(\ox)\dmu.
\end{align*}
If in addition the functions $\ph_t(\cdot)$ are ${\cal C}^1$-smooth on $U$ for a.e.\ $t\in T$, then the expected  functional $\Intf{\varphi}$ is ${\cal C}^1$-smooth on $U$.
\end{lemma}
\begin{proof} Observe first that the measurability  of the function $t\to\nabla \varphi_t(\ox)$ follows from the fact that this function can be written as the pointwise limit of a sequence of measurable functions, while the integrability follows from the estimate $\|\nabla\varphi_t(\ox)\|\le K(t)$ for a.e. $t\in T$. Now pick any sequence $x_k\overset{U}{\to}\ox$ as $k\to\infty$ and consider the sequence of the nonnegative functions
$$
f_k(t):= \bigg| \frac{\varphi_t(x_k) - \varphi_t(\ox) - \langle \nabla \varphi_t(\ox), x_k -\ox\rangle }{\| x_k-\ox\|} \bigg|.
$$
We clearly have that $f_k(t)\to 0$ as $k\to\infty$ for a.e. $t\in T$ while satisfying the upper estimate $f_k(t)\leq 2K(t)$ for such $t$. Then Lebesgue's dominated convergence theorem (see, e.g., \cite[Theorem~2.8.1]{bog}) tells us that
$$
\lim\limits_{k\to\infty}\left|\frac{\Intf{\varphi}(x_k)-\Intf{\varphi}(\ox) - \int_{\T} \langle \nabla \varphi_t(\ox), x_k-\ox \rangle \dmu }{\| x_k - \ox\|} \right|=0,
$$
which yields therefore the Fr\'echet differentiability of $\Intf{\varphi}$ at $\ox$.

It follows from the above that $\Intf{\varphi}$ is Fr\'echet differentiable on $U$. Finally, we verify that the derivative $\nabla\Intf{\varphi}$ is continuous on $U$. Indeed, take any sequence $x_k \overset{ U }{\to}\ox$ and observe that
$$
\lim_{k\to\infty}\left\|\nabla \varphi_t(\ox)-\nabla \varphi_t(x_k)\right\|= 0\;\mbox{ and }\|\nabla\varphi_t(\ox)-\nabla\varphi_t(x_k)\|\le 2K(t)
$$
for a.e. $t\in T$. Using again Lebesgue's dominated convergence theorem yields
$$\|\nabla\Intf{\varphi}(\ox)-\nabla\Intf{\varphi}(x_k)\|{\longrightarrow}0\;\text{ as  }\;k\to \infty,
$$
which completes the proof of the lemma.
\end{proof}

The second lemma establishes a {\em strong approximation} of $p$-robust minimizers of constrained expected-integral functionals by $\ve$-minimizers of a sequence of regularized unconstrained  $p$-dependent integral functionals.

\begin{lemma}[\bf approximations of $p$-robust minimizers]\label{Robustlocalminimizer} Let $p\in[1,\infty)$, and let $(\ox,\ooy) $ be a $p$-robust minimizer of the expected-integral functional \eqref{eif} on the closed set $B\times C$, where $\varphi\colon\T\times\X\times\Y\to\Rex$ is a normal integrand. Given a sequence $\varepsilon_k\dn 0$ as $k\to\infty$, let $\{(x_k,\x_k,\y_k)\big\}$ be a sequence of $\varepsilon_k$-minimizers of the function $\Psi_k\colon\X\times\Leb^1({\T},\X)\times\Leb^1(T,\Y)\to\Rex$ defined by
\begin{align}\label{functionPsi}
\Psi_k(u,\v,\w):=\IntfLp{\varphi}(\v,\w)+ k (\| \v - u\|_p)^p+ \|\ox-u\|^p+\delta_{B\times C}(u,\w),
\end{align}
where $\IntfLp{\varphi}$ is defined in \eqref{eifLp}.
Then the following assertions hold:
\begin{enumerate}[label={(\alph*)},ref={(\alph*)}]

\item[\bf(i)] $\Psi_k$ is l.s.c.\ and bounded from below on $\X\times\Leb^1({\T},\X)\times\Leb^1(T,\Y)$.

\item[\bf(ii)] $k(\|\x_k-x_k\|_p)^p\to 0$ and $\|\x_k-\ox\|_p\to 0$ as $k\to\infty$.

\item[\bf(iii)] $\IntfLp{\varphi}(\x_k,\y_k)\rightarrow \Intf{\varphi}(\ox,\ooy)$  as $k\to\infty$.
\end{enumerate}
In particular, we have the representation of $\Intf{\varphi}$ at $(\ox,\ooy)$ as follows:
\begin{equation}\label{igualdadproblemasperturbados}
\Intf{\varphi}(\ox,\ooy)=\sup\limits_{k\in\N}\inf\left\{\Psi_k(u,\v,\w)\;\bigg|\;\begin{array}{c}
u\in\X,\;\v\in\Leb^1({\T},\X),\\
\w\in\Leb^1({\T},\Y)
\end{array} \ \right\}.
\end{equation}
\end{lemma}
\begin{proof} Observe first that assertion (i) is an easy consequence of Fatou's lemma, the lower semicontinuity of the integrand in \eqref{functionPsi}, and the lower growth condition \eqref{lower_bound_assump}. To verify (ii) and (iii) simultaneously, fix $k\in\N$ and $\gamma>0$ and then define the numbers
\begin{align*}
\iota_k&:=\inf\left\{\Psi_k(u,\v,\w)\;\Big|\;\begin{array}{c}
u\in\X,\;\v\in\Leb^1({\T},\X),\\
\w\in\Leb^1({\T},\Y)
\end{array} \ \right\},\\
\kappa_\gamma&:=\inf\left\{\int_{\T}\varphi_t\big(\v(t),\w(t)\big)\dmu\;\Bigg|\;\begin{array}{c}
\disp\int_{\T}\| \v(t)-u\|^p\dmu\le\gamma,\;u\in B,\\
\v\in\Leb^p({\T},\X), \text{ and }\;\w\in C
\end{array}\right\}.
\end{align*}
It follows from the lower growth condition \eqref{lower_bound_assump} that $\iota_k>-\infty $ and $\kappa_\gamma>-\infty$.  Since $(x_k,\x_k,\y_k)$ is a $\varepsilon_k$-minimizer of $\Psi_k$ (i.e., $\Psi_k(x_k,\x_k,\y_k)\le\iota_k+\epsilon_k$), and since $\iota_k\le\Intf{\varphi}(\ox,\ooy)$, we have by \eqref{lower_bound_assump} that
\begin{align*}
k\int_T\|\x_k(t)-x_k\|^p\dmu&\le\Psi_k(x_k,\x_k,\y_k) +\int_T\nu(t)\dmu\\
&\le\iota_k+\epsilon_k +\int_T\nu(t)\dmu\\
&\le\Intf{\varphi}(\ox,\ooy)+\varepsilon_1+\int_T\nu(t)\dmu<\infty,
\end{align*}
where the last inequality ensures that $\x_k\in\Leb^p(T,\X)$. Furthermore, dividing the last inequality above by $k$, we have that $\int_{\T}\|\x_k(t)-x_k\|^p\dmu\to 0$. Denoting further $\eta_k:=\int_{\T}\|\x_k(t)-x_k\|^p\dmu $ and observing that the triple $(x_k,\x_k,\y_k )$ satisfies the conditions above while taking the infimum in the definition of $\kappa_{\eta_k}$, we arrive at the estimates
\begin{equation}\label{eq:00}
\begin{array}{ll}
\kappa_{\eta_k}-\varepsilon_k&\le\disp\int_{\T}\varphi_t\big(\x_k(t),\y_k(t)\big)\dmu-\varepsilon_k\\
&\le \IntfLp{\varphi}(\x_k,\y_k)+ k (\| \x_k - x_k\|_p)^p+ \|\ox-x_k\|^p-\varepsilon_k\\
&\leq \Psi_k(x_k,\x_k,\y_k)-\varepsilon_k\\
&\le\iota_k\le\Intf{\varphi}(\ox,\ooy).
\end{array}
\end{equation}
Since $(\ox,\ooy)$ is  a $p$-robust minimizer on $B\times C$, we have that $\kappa_{\eta_k} \to \Intf{\varphi}(\ox,\ooy)$. Hence passing in \eqref{eq:00} to the limit as $k\to\infty$ gives us the convergence
\begin{equation*}
\IntfLp{\varphi}(\x_k,\y_k)  \longrightarrow \Intf{\varphi}(\ox,\ooy),\quad k\to\infty,
\end{equation*}
which implies by using \eqref{eq:00} that assertions (ii) and (iii) of the lemma hold together with the representation of $\Intf{\varphi}(\ox,\ooy)$ in \eqref{igualdadproblemasperturbados}.
\end{proof}

It is worth mentioning that all the assertions of Lemma~\ref{Robustlocalminimizer}, expect the lower semicontinuity of $\Psi_k$, remain true if we replace the lower boundedness condition  \eqref{lower_bound_assump} by the weaker assumption that the integral functional $\IntfLp{\varphi}$ from \eqref{eifLp} is bounded from below on the set $\Leb^{p}(T,\X)\times C$. Furthermore, the following example shows that the latter condition is necessary to have $p$-robust minimizers of expected-integral functionals of the type $\Intf{\varphi}$.

\begin{example} Let $[0,1]$ be equipped with the Lebesgue measure, let the normal integrand $\varphi: [0,1] \times \R\times \R \to \R$ be given by
$$
\varphi_t(x,y):=|y|-\exp(|x|),
$$
and let $B$ and  $C$ be the unit closed ball in $\X$ and $\Leb^{1}(T,\Y)$, respectively. Then it is easy to see that $\inf_{B\times C}\Intf{\varphi}\in\R$, but $\wedge_{p,B\times C}\Intf{\varphi} =-\infty$, which tells us that the lower boundedness condition \eqref{lower_bound_assump} is necessary for the $p$-robustness of the infimum of $\Intf{\varphi}$ in the sense of Definition~{\rm\ref{ROBUSTEDLOCALMINIMUM}}.
\end{example}

Now we are ready to establish a major result, which provides sequential necessary optimality conditions for robust minimizers of expected-integral functionals. The obtained result is certainly of its independent interest while playing a crucial role in deriving generalized Leibniz rules in the next section.

\begin{theorem}[\bf sequential necessary conditions for robust minimizers]\label{aproximatecalculsrules} Let $p,q\in(1,\infty)$ be such that $1/p+1/q=1$, and let $(\ox,\ooy)\in\X\times\Leb^1(T,\Y)$ be a $p$-robust local minimizer of the expected-integral functional $\Intf{\varphi}$ generated in \eqref{eif} by a normal integrand $\varphi\colon\T\times\X\times\Y\to\Rex$. Then there exist sequences $\{x_k\}\subset\X$, $\{\x_k\}\subset\Leb^p({T},\X)$, $\{\x_k^*\}\subset\Leb^q({T},\X)$, $\{\y_k\}\subset\Leb^1({T},\Y)$, and $\{\y_k^*\}\subset\Leb^\infty({T},\Y)$ satisfying the following conditions:
\begin{enumerate}[label=\alph*),ref=\alph*)]

\item[\bf(i)] $\big(\x_k^*(t),\y_k^\ast(t)\big)\in\Hat\partial\varphi_t\big(\x_k(t),\y_k(t)\big)$ for a.e.\ $t\in T$ and all $k\in\N$.

$\;$\item[\bf(ii)] $\|\ox-x_k\|\to 0$, $\|\ox-\x_k\|_p\to 0$, and $\|\ooy-\y_k\|_1\to 0$ as $k\to\infty$.

$\;\;\;\;$\item[\bf(iii)] $\left\|\int_T \x_k^*(t)\dmu\right\|\to 0$ and $\|\y_k^\ast\|_\infty\to 0$ as $k\to\infty$.

$\;$\item[\bf(iv)] $\|\x_k^*\|_q\|\x_k-x_k\|_p\to 0$ as $k\to\infty$.

\item[\bf(v)] $\disp\int_T\big|\varphi_t\big(\x_k(t),\y_k(t)\big)-\varphi_t\big(\ox,\ooy(t)\big)\big|\dmu\to 0$ as $k\to\infty$.
\end{enumerate}
\end{theorem}
\begin{proof}  We first prove the {\em limiting subdifferential version} of the theorem, where $\Hat\partial\ph_t$ is replaced by $\partial\ph_t$. This subdifferential replacement allows to employ the {\em well-developed calculus} for the limiting subdifferential \eqref{lsub} that is not available for its regular counterpart \eqref{rsub}. Then we pass to the claimed conclusions of the theorem formulated in terms of {\em regular subgradients} by using the fact that they approximate the limiting ones.

To begin with, recall that the function $x\to\|x\|^p$ with $p>1$ is continuously differentiable on $\R^n$ and its gradient satisfies the the estimates
\begin{align}\label{ineq_normp}
\|x\|^{p-1}\le\|\nabla(\|\cdot\|^p)(x)\|\le p\|x\|^{p-1} \text{ for all }\;x\in\X.
\end{align}
Picking $r\in(0,1)$ such that $(\ox,\ooy)$ is a $p$-robust minimizer on $\mathbb{B}_r(\ox)\times\mathbb{B}_r(\ooy)$, consider the function $\Psi_k$ defined in \eqref{functionPsi} with $B\times C=\mathbb{B}_r(\ox)\times\mathbb{B}_r(\ooy)$. Then representation \eqref{igualdadproblemasperturbados} from Lemma~\ref{Robustlocalminimizer} tells us that
\begin{equation*}
\Intf{\varphi}(\ox,\ooy)=\sup\limits_{k\in\N}\inf\left\{\Psi_k(u,\v,\w)\;\bigg|\;\begin{array}{c}
u\in\X,\;\v\in\Leb^1({\T},\X),\\
\w\in\Leb^1({\T},\Y)
\end{array} \ \right\}.
\end{equation*}
For all $k\in\N$, define now the positive numbers
$$
\varepsilon_k:=\left( \Intf{\varphi}(\ox,\ooy)- \inf\left\{\Psi_k(u,\v,\w)\;\bigg|\;\begin{array}{c}
u\in\X,\;\v\in\Leb^1({\T},\X),\\
\w\in\Leb^1({\T},\Y)
\end{array} \ \right\} +  \frac{1}{k} \right)^{1/2}
$$
and observe that $\varepsilon_k\dn 0$ as $k\to\infty$. Since $\Psi_k(\ox,\ox,\ooy)=\Intf{\varphi}(\ox,\ooy)$, we have that the triple $(\ox,\ox,\ooy)$ is an $\varepsilon^2_k$-minimizer of $\Psi_k$ from \eqref{functionPsi} and assume without loss of generality that $\varepsilon_k\in(0,r/2)$ for all $k\in\N$. Considering further the Banach space $\Ze:=\X\times\Leb^1({T},\X)\times\Leb^1(T,\Y)$ with the norm
\begin{align*}
\|( u,\v,\w)\|_{\Ze}:=\| u\|+\int_{{\T}}\|\v\|\dmu+\int_{{\T}}\|\w\|\dmu,
\end{align*}
we split the rest of the proof into the  following five claims.\\[1ex]
{\bf Claim~1:} {\em There exists a sequence $(x_k,\x_k,\y_k)\in\X\times\Leb^p(T,\X)\times\Leb^1(T,\Y)$ satisfying the following conditions:\vspace*{-0.05in}
\begin{enumerate}
\item[\bf(a)] $\|(\ox,\ox,\ooy)-(x_k,\x_k,\y_k)\|_{\Ze}\le\epsilon_k$ for all $k\in\N$.
\item[\bf(b)] $\Psi_k(x_k,\x_k,\y_k)+\epsilon_k\|(\ox,\ox,\ooy)-(x_k,\x_k,\y_k)\|_{\Ze}\le\Psi_k(\ox,\ox,\ooy)=\Intf{\varphi}(\ox,\ooy)$.
\item[\bf(c)] The function $\Psi_k+\epsilon_k\|\cdot-(x_k,\x_k,\y_k)\|_{\Ze}$ attains its minimum at $(x_k,\x_k,\y_k)$.
\end{enumerate}}\vspace*{-0.05in}
Indeed, applying the fundamental Ekeland variational principle (see, e.g., \cite[Theorem~2.26]{m06}) to the function $\Psi_k$ at its $\epsilon^2_k$-minimizer $(\ox,\ox,\ooy)$ over the Banach space $(\Ze,\|\cdot\|_{\Ze})$ gives us for each $k\in\N$ points $(x_k,\x_k,\y_k)\in\Ze$ satisfying assertions (a), (b), and (c) of this claim. It follows from (b) and the definition of $\Psi_k$ in \eqref{functionPsi} that we have $\x_k\in\Leb^p(T,\X)$.\\[1ex]
{\bf Claim~2:} {\em The sequence $(x_k,\x_k,\y_k)$ satisfies assertions {\rm (ii)} and {\rm(v)} of the theorem.} To verify this claim, we use the choice of the triple $(x_k,\x_k,\y_k)$ as an $\epsilon_k^2$-minimizer of $\Psi_k$ and deduce from Proposition~\ref{Robustlocalminimizer} that
\begin{equation*}
k\big(\|\x_k-x_k\|_p\big)^p\to 0\;\mbox{and}\;\IntfLp{\varphi}(\x_k,\y_k) \to  \Intf{\varphi}(\ox,\ooy) \;\mbox{as}\;k\to\infty.
\end{equation*}
It follows from Claim~1(a) that $\|x-x_k\|\to 0$ as $k\to\infty$ and $(\x_k,\y_k)\to(\ox,\ooy)$ in $\Leb^1(T,\X)\times \Leb^1(T,\Y)$ as $k\to\infty$. Thus Proposition~\ref{lemma:convergenciaL1} tells us that
\begin{align*}
\displaystyle\int_T\big|\varphi_t\big(\x_k(t),\y_k(t)\big)-\varphi_t\big(\ox,\ooy(t)\big)\big|\dmu \to 0\;\mbox{ as }\;k\to\infty,
\end{align*}
which finishes the proof of this claim.\\[1ex]
{\bf Claim~3:} {\em  Define $u^*_k(t):=\nabla \|\cdot\|^p(\x_k(t)-x_k)$ on $T$. Then for a.e.\ $t\in T$ we have the relationships
\begin{align}\label{equationmira}
-\big(ku^*_k(t),0\big)\in&\partial\varphi_t\big({\x}_k(t),\y_k(t)\big)+\mathbb{B}_{2\epsilon_k}(0)\times\mathbb{B}_{2\epsilon_k}(0),\\
\label{equationmirb}
\|u_k^*(t)\|^q\le&p^q\|\x_k(t)-x_k\|^p\;\mbox{ for all }\;k\in\N,
\end{align}
and $\|k\int_{T}u_k^*(t)\dmu\|\to 0$ as $k\to\infty$.} To prove this claim, observe first that the estimate in \eqref{equationmirb} follows directly from \eqref{ineq_normp}. Consider further and integrable function $\rho\in\Leb^1(T,\R_+)$ with $\int_T\rho(t)\dmu<r/2$ and define the normal integrand $h^k\colon T\times\X\times\Y\to\Rex$ by
\begin{align*}
 h_t^k(v,w) := &\varphi_t(v,w)+k\|v-x_k\|^p \\+ &\epsilon_k\left(\|v-\x_k(t)\|+\|w-\y_k(t)\|\right)+ \delta_{\B_{\rho(t)}(\y_k(t))}(w).
\end{align*}
We clearly have the equalities
\begin{equation}\label{eq:newdet}
\begin{aligned}
&\IntfLp{h^k} (\x_k,\y_k)+ \| \ox - x_k\|^p=\displaystyle\int_T h^k_t\big({\x}_k (t),{\y}_k(t)\big)\dmu+ \| \ox - x_k\|^p \\
& =  \IntfLp{\varphi}(\x_k,\y_k) + k \int_T \| \x_k(t) - x_k \|^p \dmu + \| \ox - x_k\|^p= \Psi_k(x_k,\x_k,\y_k).
\end{aligned}
\end{equation}
It follows from Claim~1(b) that $\IntfLp{h^k} (\x_k,\y_k)$ is finite. Employing \eqref{eq:newdet} and Claim~1(b) with $u=x_k$, we get for all $ \v\in\Leb^1({T},\X)$ and $ \w\in\Leb^1({T},\Y)$ that
\begin{equation}\label{eq2newver}
\begin{array}{ll}
\Psi_k(x_k,\x_k,\y_k) \leq\Psi_k( x_k,\v, \w)  +  \epsilon_k\| ( x_k,\v,\w)  -(x_k,\x_k,\y_k)\|_{\Ze}\\
= \disp\int_T \varphi_t\big(\v(t),\w(t)\big) \dmu + k\int_T\| \x_k(t) - x_k \|^p \dmu + \| \ox - x_k\|^p\\
\disp+\epsilon_k \left(  \int_T \| \x_k(t) - \v(t)\|  \dmu  + \int_T \| \y_k(t) - \w(t)\|  \dmu    \right)\\
\leq\IntfLp{h^k}(\v,\w) + \| \ox - x_k\|^p,
\end{array}
\end{equation}
where in the last inequality uses the inclusion
$$
\big\{ \w \in \Leb^1(T,\Y)\;\big|\;\w(t) \in \mathbb{B}_{\rho(t)} (\y_k(t)) \text{ a.e. }\big\}\subset\mathbb{B}_{r}(\ooy) =C.
$$
Then we deduce from  \eqref{eq:newdet} and \eqref{eq2newver} that
\begin{equation*}
\begin{array}{ll}
\displaystyle\int_T h^k_t\big({\x}_k (t),{\y}_k(t)\big)\dmu=\inf\left\{\int_T h^k_t\big(\v(t),\w(t)\big)\dmu\;\Big|\;\begin{array}{c}
\v\in\Leb^1({T},\X)\\
\w\in\Leb^1({T},\Y)
\end{array}\right\}\\
=\disp\int_T\inf\{ h_t^k(v,w): v\in\X,\;w\in\Y \} \dmu,
\end{array}
\end{equation*}
where the last equality is due to Proposition~\ref{dualityRock}. Since the inequality $$ h^k_t\big({\x}_k (t),{\y}_k(t)\big)\geq   \inf_{v\in\X,\;w\in\Y}h_t^k(v,w)$$ always holds, this implies that  for a.e.\ $t\in T$ the function $h^k_t(\cdot,\cdot)$ attains its minimum at $(\x_k(t),\y_k(t))$, and so $0\in\partial h^k_t(\x_k(t),\y_k(t))$ a.e.\ on $T$ by the subdifferential Fermat rule from \cite[Proposition~1.114]{m06}. Taking into account the summation structure of $h^k_t$ and employing the sum rule for the limiting subdifferential from \cite[Theorem~2.33(c)]{m06}, we verify the fulfillment of \eqref{equationmira} for a.e.\ $t\in T$. It follows further from Claim~1(c) by using $(\v,\w)=(\x_k,\y_k)$ that
\begin{align*}
&k\int_T\|\x_k(t)-x_k\|^p\dmu+\|x_k-\ox\|^p+\epsilon_k\|x_k-\ox\|\\
&=\inf\limits_{u\in\X}\left( k\int_T\|\x_k(t)-u\|^p\dmu+\|u-\ox\|^p+\epsilon_k\|u-\ox\|\right).
\end{align*}
Hence Lemma~\ref{lemmadifferenciabilidadnorma} ensures the differentiability of the function
\begin{align*}
u\to\int_T\|\x_k(t)-u\|^p\dmu.
\end{align*}
Applying again the subdifferential Fermat rule and the direct calculation by Lemma~\ref{lemmadifferenciabilidadnorma} brings us to the convergence $\|k\int_{T}u_k^*(t)\dmu\|\to 0$ as $k\to\infty$, which ends the verification of this claim.\\[1ex]
{\bf Claim~4:} {\em There exist sequences $\{\x_k^*\}\subset\Leb^q({T},\X)$ and $\{\y_k^*\}\subset\Leb^\infty({T},\Y)$ satisfying assertions {\rm(iii)} and {\rm(iv)} as well as assertion {\rm(i)} with the replacement of the regular subdifferential by its limiting counterpart.} Indeed, the measurability of all the functions in \eqref{equationmira} and the measurable selection theorem from \cite[Theorem~14.16]{rw} ensure the existence of measurable selections $(\x_k^\ast(t),\y_k^\ast(t))\in\partial\varphi_t( \x_k(t),\y_k(t))$ for a.e.\ $t\in T$ such that
\begin{align}\label{equation:7}
\big(\x_k^\ast(t),\y_k^\ast(t)\big)-\big(ku^*_k(t),0\big)\in 2\epsilon_k\mathbb{B}_{}(0,0)\;\text{ for a.e. }\;t\in T;
\end{align}
thus we get (i). Let us show that the sequence of quadruplets $(x_k,\x_k,\y_k,\x_k^\ast,\y^\ast_k)$ satisfies assertions (iii) and (iv). Indeed, it follows from \eqref{equationmirb} and \eqref{equation:7} that
\begin{align*}
\left(\int_{T}\|\x_k^*(t)\|^q\right)^{1/q}\le&k\left(\int_{T}\|u_k^*(t)\|^q\right)^{1/q} +2\epsilon_k\mu(T),\\
\|\y_k^\ast\|_\infty\le&2\epsilon_k,\\
\bigg\|\int_{T} x_k^*(t)\dmu\bigg\|\le&\bigg\|k\int_{T}u_k^*(t)\dmu\bigg\|+2\epsilon_k\mu(T)^{1/q}.
\end{align*}
This shows that $\x_k^\ast\in\Leb^p(T,\X)$ and $\y_k^\ast\in\Leb^\infty(T,\Y)$, and therefore assertion (iii) is verified. Furthermore, we have the relationships
\begin{align*}
\|\x_k^*\|_q\|\x_k-x_k\|_p\le&\left(k\left(\int_{\T}\|u_k^\ast(t)\|^q\dmu\right)^{1/q}+2\epsilon_k\mu(T)^{1/q}\right)\|x^k_\infty-x_k\|_p\\
\le&kp\big(\|\x_k-x_k\|_p\big)^{p/q}\|\x_k-x_k\|_p)\\
&+2\epsilon_k\mu(T)^{1/q}\|\x_k-x_k\|_p\\
=&k p(\|x_k -y_k\|_p)^p+2\epsilon_k\mu(T)^{1/q}\cdot\|x_k -y_k\|_p\to 0,
\end{align*}
which justify assertion (iv) and thus accomplish the proof of the theorem in the case of the limiting subdifferential in (i).\\[1ex]
{\bf Claim~5:} {\em Completing the proof of theorem}. It remains to show that the fulfillment of all the assertions of the theorem for the limiting subdifferential in (i) yields all of the claimed assertions as formulated therein, i.e., in terms of the regular subdifferential in (i). To proceed, consider the sequence of $(\x_k^*(t),\y_k^\ast(t))\in\partial\varphi_t(\x_k(t),\y_k(t))$ taken from Claim~4. For any $\epsilon>0$ define the multifunction $M_k^\epsilon\colon T\tto\mathbb{R}^{2(n+m)}$ by the equivalence $(x,x^\ast,y,y^\ast)\in M_k^\epsilon(t)\Longleftrightarrow$
\begin{eqnarray*}
\begin{array}{c}
(x^\ast,y^\ast)\in\Big\{\Hat\partial\varphi_t(x,y)\;\Big|\;|\varphi_t(x,y)-\varphi_t(\x_k(t),\;\y_k(t))|\le\epsilon,\;\|x-\x_k(t)\|\le\epsilon,\\
\qquad\|x^\ast-\x^\ast_k(t)\|\le\epsilon,\;\|y-\y_k(t)\|\le\epsilon,\;\|y^\ast-\y^\ast_k(t)\|\le\epsilon\Big\}.
\end{array}
\end{eqnarray*}
It follows from Theorem~\ref{lemma_measurability_reg_sub} that the defined multifunctions $M_k^\epsilon$ have measurable graphs. Furthermore, the sets $M_k^\epsilon(t)$ are nonempty for a.e.\ $t\in T$ and all large $k\in\N$ due to the definitions of the subdifferentials \eqref{rsub}, \eqref{lsub} and of the limiting normal cone \eqref{lnc}. Then the measurable selection theorem taken from Proposition~\ref{Prop_measurableselection} ensures the existences of sequences satisfying the claimed conclusions of the theorem, and thus we complete the proof.
\end{proof}\vspace*{-0.3in}

\section{Sequential Leibniz Rules for Expected-Integral Functionals}\label{Section:FuzzyCalculus}\vspace*{-0.05in}

The final section establishes two main results, which provide sequential versions of the generalized Leibniz rule for expected-integral functionals.\vspace*{0.03in}

Recall \cite{bog} that for a finite measure space $(T,\mathcal{A},\mu)$ there exist measurable disjoint sets $T_{pa}$ and $T_{na}$ such that $\mu_{pa}(\cdot):=\mu(\cdot\cap T_{pa})$ is {\em purely atomic} with countably many disjoint atoms, while $\mu_{na}(\cdot):=\mu(\cdot\cap T_{na})$ is {\em nonatomic}.

Throughout this section we assume that at a given point of interest $(\ox,\ooy)\in\dom\Intf{\varphi}$ there exists $\rho>0$ such that
\begin{align}\label{convex_cond}
\varphi_t(v,\cdot)\;\text{ is convex whenever }\;v\in\mathbb{B}_\rho(\ox)\;\text{ and }\;t\in T_{na}.
\end{align}
Note that the imposed technical assumption \eqref{convex_cond} is used in what follows to get the strong-weak lower semicontinuity of the integral functional under consideration in $L^1\times L^1$. It allows us to obtain more attractive representations of the generalized Leibniz rules obtained below. In some settings this assumption can be dismissed due to the Lyapunov-Aumann convexity theorem (see, e.g., \cite{aum,bog,m06}), while we skip here more detailed discussions on this topic.\vspace*{0.05in}

Before deriving the main results on sequential Leibniz rules, we first establish a relationship between robust local minima and conventional local minima of expected-integral functionals.  The following theorem is certainly of its independent interest, while it is needed below for deriving our main results.

\begin{theorem}[\bf robust vs.\ conventional local minima]\label{Prop_Suf_Con01} Let $p\in[1,\infty)$, and let $(\ox,\ooy)\in\dom\Intf{\varphi}$, where the normal integrand $\varphi\colon\T\times\X\times\Y\to\Rex$ satisfies the lower growth condition \eqref{lower_bound_assump}. Then  we have under the fulfilment of assumption \eqref{convex_cond} that for every $r\in(0,\rho)$
\begin{align}\label{Prop_Suf_Con01:equation}
\wedge_{p,\mathbb{B}_r(\ox)\times\mathbb{B}_r(\ooy)}\Intf{\varphi}=\inf_{\mathbb{B}_r(\ox)\times\mathbb{B}_r(\ooy)}\Intf{\varphi}.
\end{align}
\end{theorem}
\begin{proof} It is obvious that the inequality ``$\le$" always holds in \eqref{Prop_Suf_Con01:equation}. To verify the opposite one, we denote $B:=\mathbb{B}_r(\ox)$, $C:=\mathbb{B}_r(\ooy)$ and consider the value $I:=\inf_{B\times C}\Intf{\varphi}$, which is assumed to be finite without loss of generality; otherwise the conclusion is trivial. For each $k\in\N$ define the number
\begin{equation*}
\iota_k:=\inf\left\{\int_{\T}\varphi_t\big(\v(t),\w(t)\big)\dmu\;\Bigg|\;\begin{array}{c}
\disp\int_{\T}\|\v(t)-u\|^p\dmu\le 1/k,\;u\in B,\\
\w\in\Leb^1({\T},\X)\;\text{ and }\;\w\in C
\end{array}\right\}
\end{equation*}
and take sequences $\varepsilon_k\dn 0$ and $(x_k,\x_k,\y_k)\in B\times\textnormal{L}^p({T},\X)\times C$ with $\int_{\T}\|\x_k(t)-x_k\|^p\dmu\le 1/k$ such that
\begin{equation}\label{Prop:SuficientCondition_Eq01}
-\varepsilon_k+\int_{{T}}\varphi_t\big(\x_k(t),\y_k(t)\big)\dmu\le\iota_k\le I.
\end{equation}
The compactness of $B$ and standard real analysis allow us to select subsequences (without relabeling) such that $x_k\to\bar u\in B$, $\int_T\|\x_k(t)-\bar u\|^p\dmu\to 0$, and hence
$\x_k(t)\to\bar u$ as $k\to\infty$ for almost all $t\in T$.

We split the rest of the proof into the following three steps.\\[1ex]
{\bf Claim~1:} {\em There exist a subsequence $\{\y_{k_j}\}$, a function $\y\in C$, and a decreasing sequence of sets $\{A_s^1\}_{s\in\N}\subset\mathcal{A}$ with $\mu(A_s^1)\to 0$ as $s\to\infty$ such that for each $s\in\N$ the functions $\y_{k_j}\1_{T\backslash A_s^1}$ converge weakly in $\Leb^1(T,\Y)$ to $\y\1_{T\backslash A_s^1}$ as $j\to\infty$}. Indeed, by \cite[Theorem~4.7.23]{bog} we find a subsequence $\{\y_{k_j}\}$, a function $\y\in\Leb^1(T,\Y)$, and a decreasing sequence of sets $\{A_s^1\}_{s\in\N}\subset\mathcal{A}$ with $\mu(A_s^1)\to 0$ as $s\to\infty$ for which $\y_{k_j}\1_{T\backslash A_m^1}$ converge weakly in $\Leb^1(T,\Y)$ to $\y\1_{T\backslash A_s^1}$ whenever $s\in\N$. To show that $\y\in C$, observe that for each $s\in\N$ we have that the functions $\y_{k_j}\1_{T\backslash A_s^1}+\ooy\1_{A_s^1}$ belong to $C$ and converge weakly in $\Leb^1(T,\Y)$ to $\y\1_{T\backslash A_s^1}+\ooy\1_{A_s^1}$. By the convexity and closedness of $C$ in $\Leb^1(T,\Y)$ we have that $\y\1_{T\backslash A_s^1}+\ooy\1_{A_s^1}\in C$. Thus it follows from the measure convergence $\mu(A_s^1)\to 0$ as $s\to\infty$ that the functions $\y\1_{T\backslash A_s^1}+\ooy\1_{A_s^1}$ converge to $\y$ as $s\to\infty$ in the norm topology of $\Leb^1(T,\Y)$. By the closedness of $C$ we conclude that $\y\in C$ and hence verify all the statements of this claim.\\[1ex]
{\bf Claim~2:} {\em There exits a decreasing sequence of sets $\{A_s^2\}_{s\in\N}\subset\mathcal{A}$ such that $\mu(A^2_s)\to 0$ as $s\to\infty$ and the functions $\x_{k_j}$ converge to $\bar u$ uniformly on $T\backslash A^2_s$ as $j\to\infty$}. To verify this claim, we employ the Egorov theorem (see, e.g., \cite[Theorem~2.2.1]{bog}) and find a decreasing sequence of sets $\{A_s^2\}\subset\mathcal{A}$ with $\mu(A^2_s)\to 0$ as $s\to\infty$ as well as a subsequence $\{\x_{k_j}\}$, which converges to $\bar u$ uniformly on $T\backslash A^2_s$ as $j\to\infty$. This readily justifies Claim~2.\\[1ex]
To proceed further, fix $s\in\N$ and form the set $D_s:=A^1_s\cup A_s^2$ and $D^c_s$ standing for its complement. Then pick $j_s$ such that $\x_{k_j}(t)\in\mathbb{B}_r(\ox)$ for all $t\in T\backslash A^2_s$ and all $j\ge j_s$. For all such $j$ define the functions $\v^s_j:=\x_{k_j}\1_{D_s^c}+\ox\1_{D_s}$, $\v_s:=\bar u\1_{D_s^c}+\ox\1_{D_s}$, $\w^s_j:=\y_{k_j}\1_{D_s^c}+\ooy\1_{D_s}$, and $\w_s:=\y\1_{D_s^c}+\ooy\1_{D_s}$. It follows from the constructions above that the functions $\v^s_j$ converge to $\v_s$ as $j\to\infty$ strongly in $\textnormal{L}^p({T},\X)$ and that the functions $\w^s_j$ converge to $\w_s$ as $j\to\infty$ weakly in $\Leb^1(T,\Y)$. This is used in the proof of the next claim.\\[1ex]
{\bf Claim~3:} {\em  We have that $\Intf{\varphi}(\bar u,\y)\le\wedge_{p,\mathbb{B}_r(\ox)\times\mathbb{B}_r(\ooy)}\Intf{\varphi}$, and consequently the estimates in \eqref{Prop:SuficientCondition_Eq01} are satisfied.} Indeed, it follows from the lower semicontinuity result \cite[Theorem~2.1]{bal} due to the assumptions in \eqref{lower_bound_assump} and \eqref{convex_cond} that on the measure nonatomicity set $T_{na}$ we have the inequalities
\begin{align*}
\int_{T_{na}}\varphi_t\big(\v_s(t),\w_s(t)\big)\dmu\le&\liminf\limits_{j\to\infty}\int_{T_{na}}\varphi_t\big(\v^s_j(t),\w^s_j(t)\big)\dmu\\
=&\liminf\limits_{j\to\infty}\left(\int_{D_s^c\cap{T_{na}}}
\varphi_t\big(\x_{k_j}(t),\y_{k_j}(t)\big)\dmu\right)\\&+\int_{D_s\cap{T_{na}}}\varphi_t\big(\ox,\ooy(t)\big)\dmu\\
\le& \liminf\limits_{j\to\infty} \left(\int_{{T_{na}}}\hspace{-0.3cm}\varphi_t\big(\x_{k_j}(t),\y_{k_j}(t)\big)\dmu  +\hspace{-0.1cm} \int_{D_s\cap{T_{na}}}\hspace{-0.6cm}\nu(t)\dmu\right)\\&+\int_{D_s\cap{T_{na}}}\varphi_t\big(\ox,\ooy(t)\big)\dmu
\\
\le&\liminf\limits_{j\to\infty}\left(\int_{{T_{na}}}\varphi_t\big(\x_{k_j}(t),\y_{k_j}(t)\big)\dmu\right)\\&+\int_{D_s\cap{T_{na}}}\varphi_t\big(\ox,\ooy(t)\big)\dmu+\int_{D_s\cap{T_{na}}}\hspace{-0.2cm}\nu(t)\dmu,
\end{align*}
where $\nu$ is taken from \eqref{lower_bound_assump}. Considering further the atomic measure set $T_{pa}$, we have that $\w_j$ converges pointwise as $j\to\infty$ to $\w$ on $T_{pa}$. Proceeding similarly to the above with the use of Fatou's lemma gives us the estimates
\begin{equation*}
\begin{aligned}
\int_{T_{pa}}\varphi_t\big(\v_s,\w_s(t)\big)\dmu\le&\int_{T_{pa}}\liminf\limits_{j\to\infty}\varphi_t\big(\v^s_j(t),\w^s_j(t)\big)\dmu\\
\le&\liminf\limits_{j\to\infty}\int_{T_{pa}}\varphi_t\big(\v_j^s(t),\w_k^s(t)\big)\dmu\\
\le&\liminf\limits_{j\to\infty}\left(\disp\int_{D_s^c\cap{T_{pa}}}\varphi_t\big(\x_{k_j}(t),\y_{k_j}(t)\big)\dmu\right)\\&+\int_{D_s\cap{T_{pa}}}\varphi_t\big(\ox,\ooy(t)\big)\dmu\\
\le&\liminf\limits_{j\to\infty}\left(\disp\int_{T_{pa}}\varphi_t\big(\x_{k_j}(t),\y_{k_j}(t)\big)\dmu\hspace{-0.1cm}+\hspace{-0.1cm}\int_{D_s\cap T_{pa}}\hspace{-0.7cm}\nu(t)\dmu\right)\\
&+\disp\int_{D_s\cap{T_{pa}}}\varphi_t\big(\ox,\ooy(t)\big)\dmu\\
\le&\liminf\limits_{j\to\infty}\left(\int_{ T_{pa}}\varphi_t\big(\x_{k_j}(t),\y_{k_j}(t)\big)\dmu\right)\\&+\int_{D_s\cap{T_{pa}}}\varphi_t\big(\ox,\ooy(t)\big)\dmu+\int_{D_s\cap T_{pa}}\nu(t)\dmu.
\end{aligned}
\end{equation*}
Unifying now the above estimates on the sets $T_{na}$ and $T_{pa}$, we get
\begin{align*}
\int_{T}\varphi_t\big(\v_s(t),\w_s(t)\big)\dmu\le\wedge_{p,\mathbb{B}_r(\ox)\times\mathbb{B}_r(\ooy)}\Intf{\varphi}+\int_{D_s}\hspace{-0.1cm}\varphi_t\big(\ox,\ooy(t)\big)\dmu\hspace{-0.1cm}+\hspace{-0.1cm} \int_{D_s}\hspace{-0.3cm}\nu(t)\dmu.
\end{align*}
Consequently, it gives us for all $s\in\N$ that
\begin{align*}
\int_{D_s^c}\varphi_t\big(\bar u,\y(t)\big)\dmu\le\wedge_{p,\mathbb{B}_r(\ox)\times\mathbb{B}_r(\ooy)}\Intf{\varphi}+\int_{D_s}\nu(t)\dmu.
\end{align*}
Passing finally to the limit as $s\to\infty$, we arrive at the inequality
\begin{align*}
\inf_{\mathbb{B}_r(\ox)\times\mathbb{B}_r(\ooy)}\Intf{\varphi}\le\int_{T}\varphi_t\big(\bar u,\y(t)\big)\dmu\le\wedge_{p,\mathbb{B}_r(\ox)\times\mathbb{B}_r(\ooy)}\Intf{\varphi}
\end{align*}
and thus complete the proof of the theorem.
\end{proof}

Now we are ready to derive our first sequential Leibniz rule for regular subdifferentiation of expected-integral functionals. Recall that the regular subdifferential of $\Intf{\varphi}$ at a point $(\ox,\oy)\in \dom \Intf{\varphi}$, denoted by $\Hat{\partial}\Intf{\varphi}(\ox,\ooy)$, is given in this setting as the collection of all $(\ox^\ast, \oy^\ast) \in \X \times \Leb^\infty(T,\Y)$ such that
\begin{align*}
\liminf\limits_{(x,\y)\to(\ox,\ooy)}\left( \frac{\Intf{\varphi}(x,\y)-\Intf{\varphi}(\ox,\ooy)-\langle \ox^\ast,x-\ox\rangle-\int_T\langle \oy^\ast(t),\oy-\ooy(t)\rangle\dmu   }{   \| x -\ox\|+\|\y-\ooy\|_1}\right)\geq 0.
\end{align*}

The first result corresponds to the case where the $x$-components of the measurable subgradient selections under the integral sign are taken from the space $\Leb^p({T},\X)$ with $p\in(1,\infty)$.

\begin{theorem}[\bf sequential Leibniz rule for expected-integral functionals, I]\label{theoremsubdiferential} Let $p,q\in(1,\infty)$ with $1/p+1/q=1$, and let $(\ox,\ooy)$ be a point satisfying assumption \eqref{convex_cond}. Given a regular subgradient $(\ox^*,\ooy^\ast)\in\Hat{\partial}\Intf{\varphi}(\ox,\ooy)$ of the expected-integral functional \eqref{eif} such that the function $t\to  \inf\{\varphi_t(\cdot,\cdot)-\la\ooy^\ast(t),\cdot\ra\}$ is integrable on $T$, there exist sequences $\{x_k\}\subset\X$, $\{\x_k\}\subset\Leb^p({T},\X)$, $\{\x_k^*\}\subset{\Leb}^q({T},\X)$, $\{\y_k\}\subset\Leb^1(T,\Y)$, and $\{\y_k^\ast\}\subset\Leb^\infty(T,\Y)$ such that we have the following conditions:
\begin{enumerate}[label=\alph*),ref=\alph*)]
\item[\bf(i)] $\big(\x_k^*(t),\y_k^\ast(t)\big)\in\Hat\partial\varphi_t\big(\x_k(t),\y_k(t)\big)$ for a.e.\ $t\in T$ and all $k\in\N$.
\item[\bf(ii)] $\|\ox-x_k\|\to 0$, $\|\ox-\x_k \|_p \to 0$, $ \|\ooy -\y_k \|_1 \to  0$, $\|\ooy^\ast-\y_k^\ast\|_\infty\to 0$  as $k\to\infty$.
\item[\bf(iii)] $\disp\Big\|\int_T x_k^*(t)\dmu-\ox^\ast\Big\|\to 0$,  $\|\x_k^*\|_q\|\x_k-x_k\|_p\to 0$  as $k\to\infty$.
\item[\bf(iv)] $\disp\int_T|\varphi_t\big(\x_k(t),\y_k(t)\big)-\varphi_t\big(\ox,\ooy(t)\big)|\dmu\to 0$ as $k\to\infty$.
\end{enumerate}
\end{theorem}
\begin{proof}  Without loss of generality we can assume that $\mu(T)=1$. Fix $\Tilde t\notin T$ and define the new measure space $(\Tilde{T},\tilde{\mathcal{A}},\tilde{\mathcal{\mu}})$ by $\tilde\T:=\T\cup\{\tilde t\}$, $\tilde{\mathcal{A}}$ as the smallest  $\sigma$-algebra containing  $\tilde{\mathcal{A}}$ and $\{\tilde t\}$, and $\tilde\mu\colon\tilde{\mathcal{A}}\to\R$ by
\begin{align*}
\tilde{\mu}(A):=\mu(A\backslash\{\tilde t\})+\1_{A}(\tilde t).
\end{align*}
Then extend $\ooy$ to $\tilde{T}$ by $\ooy(\tilde{t}):=0$ therein and consider the normal integrand
\begin{align}\label{phi}
\phi(t,v,w )&:=\left\{\begin{array}{cc}
\varphi_t(v,w)-\langle\ooy^\ast(t),w-\ooy(t)\rangle+\epsilon\|w-\ooy(t)\|&\text{ if }\;t\in T,\\
-\langle\ox^\ast,v-\ox\rangle+\epsilon\|v-\ox\|+\delta_{\ox+\mathbb{B}}(v)&\text{ if }\;t=\tilde{t}.
\end{array}\right.
\end{align}
Picking $\epsilon\in(0,1)$ and using $(\ox^*,\ooy^\ast)\in\Hat{\partial}\Intf{\varphi}(\ox,\ooy)$ together with condition \eqref{convex_cond}, we find $\eta\in(0,\rho)$ such that
\begin{align*}
\Intf{\phi}^{\tilde{\mu}}(u,\w)\ge\Intf{\phi}^{\tilde{\mu}}(\ox,\ooy)=\Intf{\varphi}(\ox,\ooy)\;\mbox{ for all }\;u\in\mathbb{B}_\eta( \ox),\;\w\in\mathbb{B}_{\eta}(\ooy)\subset\Leb^1(\tilde{T},\Y).
\end{align*}
Since $\phi$ satisfies the assumptions in \eqref{lower_bound_assump} and \eqref{convex_cond} at $(\ox,\ooy)$, it follows from Theorem~\ref{Prop_Suf_Con01} that
$\Intf{\phi}^{\tilde{\mu}}$ attains a $p$-robust minimum at $(\ox,\ooy)$ on $\mathbb{B}_\eta(\ox)\times\mathbb{B}_{\eta}(\ooy)$. Thus Theorem~\ref{aproximatecalculsrules} gives us $u\in\X$, $\v\in\Leb^p({\tilde T},\X)$, $\v^*\in\Leb^q({\tilde T},\X)$, $\w\in\Leb^1({\tilde T},\Y)$, and $\w^*\in\Leb^\infty({\tilde T},\Y)$ satisfying the conditions
\begin{enumerate}[leftmargin=0.4in,nolistsep,label=\roman*),ref=\roman*)]
\item[\bf(a)] $\big(\v^*(t),\w^\ast(t)\big)\in\Hat{\partial}\phi_t\big(\v(t),\w(t)\big)$ for a.e.\ $t\in\Tilde T$.
\item[\bf(b)] $\|u-\ox\|\le\epsilon$, $\|\ox-\v\|_p\le\epsilon$, $\|\ooy-\w\|_1\le\epsilon$.
\item[\bf(c)] $\left\|\disp\int_{\tilde T}\v^*(t)\tilde\mu(dt)\right\|\le\epsilon$, $\|\w^\ast\|_\infty\le\epsilon$.
\item[\bf(d)] $\|\v^*\|_q\|\v-u\|_p\le\epsilon$, $\disp\int_{\tilde T}\Big|\phi_t\big(\v(t),\w(t)\big)-\phi_t\big(\ox,\ooy(t)\big)\Big|\, \tilde\mu(dt)\le\epsilon$.
\end{enumerate}

It follows from (a), (b), and the limiting subdifferential sum rule applied to the function $\phi$ due to its structure that $\v^\ast(\tilde{t})\in\mathbb{B}_{\epsilon}(\ox^\ast)$. Take now $\gamma\in(0,\epsilon)$ satisfying $\gamma\|\v^*\|_q\le\epsilon$  and then define the multifunction $M\colon T\tto\mathbb{R}^{2(n+m)}$ by $(x,x^\ast,y,y^\ast)\in M(t)$ if and only if
\begin{equation*}
\begin{array}{ll}
(x^\ast,y^\ast)\in\Hat\partial\varphi_t(x,y),\;\big|\varphi_t(x,y)-\varphi_t\big(\v(t),\w(t)\big)\big|\le\gamma,\\
\|x-\v(t)\|\le\gamma,\;\|x^\ast-\v^\ast(t)\|\le\gamma,\\
\|y-\w(t)\|\le\gamma,\;\|y^\ast-\ooy^\ast(t)\|\le2\epsilon+\gamma.\\
\end{array}
\end{equation*}
Theorem~\ref{lemma_measurability_reg_sub} implies that the multifunction $M(\cdot)$ is graph measurable on $T$. Using further the fuzzy sum rule for regular subgradients of the function $\phi$ (see, e.g., \cite[Theorem~2.33(b)]{m06}) and its summation structure in \eqref{phi}, we conclude that the sets $M(t)$ are nonempty for a.e.\ $t\in T$. Then the measurable selection result from Proposition~\ref{Prop_measurableselection} ensures the existence of a measurable quadruple $(\x(t),\x^\ast(t),\y(t),\y^\ast(t))$ belonging to $M(t)$ for a.e.\ $t\in T$. It tells us  that this measurable selection satisfies the relationships
\begin{equation*}
\begin{array}{ll}
\big(\x^*(t),\y^\ast(t)\big)\in\Hat\partial\varphi_t\big(\x(t),\y(t)\big)\;\mbox{ for a.e. }\;t\in T,\;\;\|\ox-\x\|_p\le 2\epsilon,\\
\|\ooy-\y\|_1\le 2\epsilon,\;\Big\|\disp\int_T\x^*(t)\dmu-x^\ast\Big\|\le 3\epsilon,\;\|\y^\ast-\ooy^\ast\|_\infty\le 3\epsilon,\\
\disp\int_T\big|\varphi_t\big(\x(t),\y(t)\big)-\varphi_t\big(\ox,\ooy(t)\big)\big|\dmu\le (\|\ooy\|_\infty +3)\epsilon.
\end{array}
\end{equation*}
Finally, we estimate
\begin{align*}
	\|\x^*\|_q\|\x-u\|_p& \le\left(  \|\x^\ast - \v^\ast\|_q +   \| \v^\ast\|_q  \right)  \left( \|\x-\v\|_p + \| \v - u\|_p\right)\\
	&\le \left(  \gamma  +   \| \v\|_q  \right) \left(  \gamma   + \| \v - u\|_p\right)\\
&	\leq \epsilon^2  +\epsilon + \gamma \| \v - u\|_p + \| \v\|_q \| \v - u\|_p\\
&\leq \epsilon^2  +\epsilon + 2\epsilon^2 +\epsilon \leq 5\epsilon,
\end{align*}
which readily completes the proof of the theorem.
\end{proof}\vspace*{-0.2in}

\begin{remark}\label{point} The following explanations of the results and proof of Theorem~\ref{theoremsubdiferential} seem to be useful for the better understanding.

(i) Let us first emphasize that the {\em sequential} form of the generalized Leibniz rule of Theorem~\ref{theoremsubdiferential} is essential for the fulfillment of the obtained results and cannot be replaced by more appealing pointwise versions. This is due to the nonrobust nature of regular subgradients used here and the lack of basic calculus rules for them in finite and infinite dimensions; see \cite{m06,m18,rw} for more discussions. A simple counterexample for the failure of the pointwise counterpart of Theorem~\ref{theoremsubdiferential} is provided by the function $\ph_t(x,y):=-|x|$ defined on a probability measure space $(T,\mathcal{A},{\mathbb P})$, where we have $\Intf{\varphi}(x,\y)=-|x|$. Note that the pointwise Leibniz rules can be obtained in terms of the robust limiting subgradients \eqref{lsub}, which will be done in our future research by using the sequential results obtained here with furnishing appropriate limiting procedures.

(ii) Regarding the proof of Theorem~\ref{theoremsubdiferential}, observe that the employed one-point-extended measure technique allows us to isolate the subgradient deterministic part in an atom of the measure as in \eqref{phi}. Using this fact, we benefit from the general structure of Theorem~\ref{Prop_Suf_Con01} and get better estimates of the subgradients of the function $\phi$. In contrast, if we simply modify the integrand (without a modification of the measure space) as
$$
\phi(t,v,w )=\varphi_t(v,w)-\langle\ox^\ast,v-\ox\rangle-\langle\ooy^\ast(t),w\rangle+ \epsilon\|v-\ox\|+\epsilon\|w-\ooy(t)\|+\delta_{\ox+\mathbb{B}}(v),
$$
then the possible estimates will be with respect to measurable sequences of subgradients $(\x_k^\ast(t),\y_k^\ast(t))\in\Hat\partial\ph(t,\x_k(t),\y_k(t))$. Since the convergence of $\x_k$ is in the norm topology of $L^p$, this does not imply that $\x_k(t)$ belongs to the interior of $\ox+\mathbb{B}$ for a.e. $t\in T$, and thus we cannot employ calculus rules omitting the indicator function of this term.
\end{remark}\vspace*{-0.05in}

The second version of the sequential Leibniz rule derived below for expected-integral functionals \eqref{eif} concerns an important setting where the basic space is $\Leb^\infty(T,\X)$. This space is very useful in applications to stochastic and economic modeling, but fails to have some properties that are largely employed in variational analysis. In particular, it is not separable and not Asplund (i.e., not each of its subspace has a separable dual in contrast, e.g., to the case of reflexive Banach spaces). It has been well recognized that neither pointwise calculus holds for limiting subgradients, nor fuzzy calculus is available for regular subgradients in non-Asplund spaces; see, e.g., \cite{m06} and the references therein. Nevertheless, in what follows we establish a sequential Leibniz rule in this framework that is a major calculus result. The obtained result and its proof essentially exploit specific features of expected-integral functionals.\vspace*{0.05in}

To proceed, we first derive the following lemma.

\begin{lemma}[\bf measurable selections of regular subgradient mappings]\label{lemmasubdeltaminimun} Let $\v(\cdot),\w(\cdot)$ be two measurable functions with values in $\X$ and $\Y$, respectively,  and let $\varepsilon(\cdot)$, $\lambda_1(\cdot)$, and $\lambda_2(\cdot)$ be three strictly positive measurable functions on $T$. Suppose that the pair $(\v(t),\w(t))$ is an $\varepsilon(t)$-minimizer of the normal integrand $\varphi_t\colon\R^n\times\R^m\to\oR$ for a.e.\ $t\in T$. Then there exist measurable functions $(\x,\x^\ast,\y,\y^\ast)$ such that for a.e.\ $t\in T$ we have the conditions
\begin{equation*}
\begin{array}{ll}
\big(\x^\ast(t),\y^*(t)\big)\in\Hat{\partial}\varphi_t\big(\x(t),\y(t)\big),\;\|\x(t)-\v(t)\|\le\lambda_1(t),\;\|\y(t)-\w(t)\|\le\lambda_2(t),\\
\disp\|\x^\ast(t)\|\le 2\frac{\epsilon(t)}{\lambda_1(t)},\;\|\y^\ast(t)\|\le 2\frac{\epsilon(t)}{\lambda_2(t)},\;\big|\varphi_t\big(\v(t),\w(t)\big)-\varphi_t\big(\x(t),\y(t)\big)\big|\le\epsilon(t).
\end{array}
\end{equation*}
\end{lemma}
\begin{proof} Define the graph measurable multifunction $M\colon T\tto\mathbb{R}^{2(n+m)}$ as follows: $(x,x^\ast,y,y^\ast)\in M(t)$ if and only if
\begin{equation}
\begin{aligned}
\begin{array}{ll}
(x^\ast,y^*)\in\Hat{\partial}\varphi_t(x,y),\\
\|x-\v(t)\|\le\lambda_1(t),\;\|y-\w(t)\|\le\lambda_2(t),\\
\disp\|x^\ast\|\le 2\frac{\epsilon(t)}{\lambda_1(t)},\;\|\y^\ast\|\le 2\frac{\epsilon(t)}{\lambda_2(t)},\\
\big|\varphi_t\big(\v(t),\w(t)\big)-\varphi_t(x,y)\big|\le\epsilon(t).
\end{array}
\end{aligned}
\end{equation}
To show that $M(t)\ne\emp$ for a.e.\ $t\in T$, fix $t\in T$ and consider the function
\begin{align*}
\psi_t(v,w):=\varphi_t(v,w)+\frac{\epsilon(t)}{\lambda^2_1(t)}\|v-\v(t)\|^2+\frac{\epsilon(t)}{\lambda^2_2(t)}\|w-\w(t)\|^2.
\end{align*}
It follows from the lower semicontinuity of $\ph_t(\cdot,\cdot)$ and the lower growth condition \eqref{lower_bound_assump} that the function $\psi_t$  attains its local minimum at some $(x_t,y_t)\in\X\times\Y$. Then it is easy to deduce from the structure of $\psi_t$ that
\begin{equation*}
\|x_t-\v(t)\|\le\lambda_1(t),\;\|y_t-\w(t)\|\le\lambda_2(t),\;\big|\varphi_t(x_t,y_t)-\varphi_t\big(\v(t),\w(t)\big)\big|\le\epsilon(t).
\end{equation*}
Furthermore, by the Fermat rule and the aforementioned fuzzy sum rule for the regular subdifferential in finite dimensions, we deduce the existence of $(x_t^\ast,y^\ast_t)\in\Hat{\partial} \varphi_t(x_t,y_t)$ such that $\|x^\ast_t\|\le 2\epsilon(t)/\lambda_1(t)$ and $\|y_t^\ast\|\le 2\epsilon(t)/\lambda_2(t)$, which implies that $M(t)$ is nonempty for a.e.\ $t\in T$. Finally, using the measurable selection result from Proposition~\ref{Prop_measurableselection} completes the proof.
\end{proof}

Now we are ready to establish the following sequential Leibniz rule with corresponding measurable selections in $\Leb^\infty(T,\X)$ and $\Leb^1(T,\X)$.

\begin{theorem}[\bf sequential Leibniz rule for expected-integral functionals, II]\label{theorem_main_fuzzy_sub} Let $(\ox,\ooy)$ be a point satisfying assumption \eqref{convex_cond} and consider $(\ox^*,\ooy^\ast)\in\Hat{\partial}\Intf{\varphi}(\ox,\ooy)$ be a regular subgradient of the expected-integral functional \eqref{eif}. Suppose that there exist $\hat{\rho}>0$ and  an integrable function $\Hat{\nu}\colon T\to(0,\infty)$ with
\begin{align}\label{growth_2}
\varphi_t(u,w)-\langle\ooy^\ast(t),w-\ooy(t)\rangle\ge-\hat{\nu}(t)\;\mbox{ as }\;u\in\mathbb{B}_{\hat{\rho}}(\ox),\;w\in\Y,\;t\in T.
\end{align}
Then there exist sequences $\{x_k\}\subset\X$, $\{\x_k\}\subset\Leb^\infty({T},\X)$, $\{\x_k^\ast\}\subset{\Leb}^1({T},\X)$, $\{\y_k\}\subset\Leb^1(T,\Y)$, and $\{\y_k^\ast\}\subset\Leb^\infty(T,\Y)$ such that
\begin{enumerate}[label=\alph*),ref=\alph*)]
\item[\bf(i)] $\big(\x_k^*(t),\y_k^\ast(t)\big)\in\Hat\partial\varphi_t\big(\x_k(t),\y_k(t)\big)$ for a.e.\ $t\in T$ and all $k\in\N$.
\item[\bf(ii)] $\|\ox-x_k\|\to 0$, $\|\ox-\x_k\|_\infty\to 0$,  $\|\ooy-\y_k\|_1\to 0$, $\|\oy^\ast - \oy_k\| \to 0$  as $k\to\infty$.
\item[\bf(iii)] $\disp\Big\|\int_T\x_k^*(t)\dmu-\ox^\ast\Big\|\to 0$,   $\disp\int_T\|\x_k^*(t)\|\cdot\|\x_k(t)-x_k\|\dmu\to 0$   as $k\to\infty$.
\item[\bf(iv)] $\disp\int_T\big|\varphi_t\big(\x_k(t),\y_k(t)\big)-\varphi_t\big(\ox,\ooy(t)\big)\big|\dmu\to 0$ as $k\to\infty.$
\end{enumerate}
\end{theorem}
\begin{proof}
To simplify the calculations, assume without loss of generality that $\ooy^\ast=0$ and that we have the condition
\begin{align}\label{lower_bound_assump2}
\varphi(t,v,w)\ge 0\;\text{ for all }\;(t,v,w)\in T\times\X\times\Y
\end{align}
by considering the shifted function given by (with no relabeling)
\begin{equation*}
\hat\varphi(t,v,w):=\varphi(t,v,w)-\langle\ooy^\ast(t),w-\ooy(t)\rangle+\hat{\nu}(t)+\delta_{\mathbb{B}_{\Hat\rho}(\ox)}(v).
\end{equation*}
We split the proof into four claims. Fix $\varepsilon\in(0,1)$ in what follows.\\[1ex]
{\bf Claim~1:} {\em There exist $x\in\X$ and integrable functions $\x\in{\Leb}^2({T},\X)$, $\x^*\in\Leb^2({T},\X)$, $\y\in\Leb^1({T},\Y)$, and  ${\y}_k^*\in{\Leb}^\infty({T},\Y)$ satisfying the conditions
\begin{enumerate}[leftmargin=0.4in,label=1.\arabic*),ref=1.\arabic*)]
\item[\bf(a1)] $\|{\x}(t)-\ox\|\le\epsilon$ for a.e.\ $t\in T$.
\item[\bf(b1)] $\mu(A)\le\epsilon$, where $A:=\big\{t\in{T}\;\big|\;\|{\x}(t)-x\|=\varepsilon\big\}$.
\item[\bf(c1)] $\big({\x}^*(t),{\y}^\ast(t)\big)\in\Hat\partial\varphi_t\big({\x}(t),{\y}(t)\big)$ for a.e.\ $t\in A^c$.
\item[\bf(d1)] $\|\ox-x\|\le\epsilon$, $\|\ooy-{\y}\|_1\le\epsilon$, $\|{\y}^\ast\|_\infty\le\epsilon$.
\item[\bf(e1)] $\disp\Big\|\int_T{\x}^*(t)\dmu-\ox^\ast\Big\|\le\epsilon$ and $\|{\x}^*\|_2\|{\x}-{x}\|_2\le\epsilon^2$.
\item[\bf(f1)] $\disp\int_T\big|\varphi_t\big({\x}(t),{\y}(t)\big)-\varphi_t\big(\ox,\ooy(t)\big)\big|\dmu\le\epsilon$.
\item[\bf(g1)] $\disp\int_{A}\varphi_t\big(\ox,\ooy(t)\big)\dmu\le\varepsilon^2$, $\disp\int_{A}\|{x}^*(t)\|\dmu\le 2\varepsilon$.
\end{enumerate}}
To verify this claim, consider the function $\phi(t,u,w):=\varphi(t,u,w)+\delta_{\mathbb{B}_\varepsilon(\ox)}(u)$ for which we clearly have $(\ox^\ast,0)\in\Hat{\partial} \Intf{{\phi}}(\ox,\ooy)$. Applying Theorem~\ref{theoremsubdiferential} to the latter function gives us sequences of vectors $x_k\in\X$ and measurable mappings $\x_k\in\textnormal{L}^2({T},\X)$,
${\x}_k^*\in\textnormal{L}^2({T},\X)$, ${\y}_k\in\textnormal{L}^1({T},\Y)$, and ${\y}_k^*\in\textnormal{L}^\infty({T},\Y)$ with
\begin{equation*}
\begin{array}{ll}
\big({\x}_k^*(t),{\y}_k^\ast(t)\big)\in\Hat\partial{\phi}_t\big({\x}_k(t),{\y}_k(t)\big)\;\mbox{ for a.e. }\;t\in T;\\
\|\ox-x_k\|\to 0,\;\disp\int_T\|\ox-{\x}_k(t)\|^2\dmu\to 0,\;\mbox{ and }\;\disp\int_T\|\ooy(t)-{\y}_k(t)\|\dmu  \to 0;\\
\Big\|\disp\int_T{\x}_k^*(t)\dmu-\ox^\ast\Big\|\to 0,\;\|{\y}_k^\ast\|_\infty\to 0,\;\mbox{ and }\;\|{\x}_k^*\|_2\|{\x}_k-{x}_k\|_2\to 0;\\
\disp\int_T\big|{\phi}_t\big({\x}_k(t),{\y}_k(t)\big)-{\phi}_t\big(\ox,\ooy(t)\big)\big|\dmu\to 0\;\mbox{ as }\;k\to\infty.
\end{array}
\end{equation*}
In particular, we have $({\x}_k^*(t),{\y_k}(t))\in\Hat{\partial}\varphi_t({\x}_k(t),{\y}(t) )$ for the original integrand $\ph$ whenever $\|{\x}_k(t)-\ox\|<\varepsilon$ as $k\in\N$. Defining further
the measurable sets $A_k:=\big\{t\in{T}\;\big|\;\|{\x}_k(t)-x\|=\varepsilon\big\}$ ensures that $\mu(A_k)\to 0$ as $k\to\infty$ by the convergence of $\{\x_k\}$ in $\textnormal{L}^2({T},\X)$. Then for all $k\in\N$ sufficiently large we get
\begin{equation*}
\begin{array}{ll}
\|\ox-x_k\|\le\varepsilon,\;\disp\int_T\|\ox-{\x}_k(t)\|^2\dmu\le\varepsilon,\;\disp\int_T\|\ooy(t)-{\y}_k(t)\|\dmu\le\varepsilon;\\
\Big\|\disp\int_T{\x}_k^*(t)\dmu-\ox^\ast\Big\|\le\varepsilon,\;\|{\y}_k^\ast\|_\infty\le\varepsilon,\;\|{\x}_k^*\|_2\|{\x}_k-{x}_k\|_2\le\varepsilon^2;\\
\disp\int_T\big|\varphi_t\big({\x}_k(t),{\y}_k(t)\big)-\varphi_t\big(\ox,\ooy(t)\big)|\dmu\le\varepsilon;\\
\disp\int_{A_k}\varphi_t\big(\ox,\ooy(t)\big)\dmu\le\varepsilon^2.
\end{array}
\end{equation*}
Remembering now the construction of the sets $A_k$ gives us the inequalities
\begin{align*}
{\varepsilon^2}&\geq\disp\int_{A_k}\|{\x}^*_k(t)\|\cdot\|{x}_k-{\x}_k(t)\|\dmu
\ge\disp\int_{A_k}\|{\x}^*_k(t)\|\big(\|\ox-{\x}_k(t)\|-\|\ox-{x}_k\|\big)\dmu\\
&\ge\disp\int_{A_k}\|{\x}^*_k(t)\|\Big(\varepsilon-\frac{\varepsilon}{2}\Big)\dmu=\frac{\varepsilon}{2}\disp\int_{A_k}\|{\x}^*_k(t)\|\dmu.
\end{align*}
Thus we arrive at the estimate
\begin{equation*}
\disp\int_{A_k}\|{x}^*_k(t)\|\dmu\le 2\varepsilon,
\end{equation*}
which implies in turn all the statements (a1)--(g1) of this claim by relabeling $A:=A_k$ and
$\big({\x}(t),{\y}(t),{\x}^*(t),{\y}(t)\big):=\big({\x}_k(t),{\y_k}(t),{\x}_k^*(t),{\y_k}(t)\big)$ on $T$.\\[1ex]
{\bf Claim~2:} {\em Defining $\varepsilon(t):=\varphi(t,\ox,\ooy(t))$, there exist measurable functions $\v$, $\v^\ast$, $\w$ and $\w^*$ such that for a.e.\ $t\in A$ we have}
\begin{enumerate}[leftmargin=0.4in,nolistsep,label=2.\arabic*),ref=2.\arabic*)]
\item[\bf(a2)] $\big(\v^\ast(t),\w^\ast(t)\big)\in\Hat\partial\varphi_t\big(\v(t),\w(t)\big)$.
\item[\bf(b2)] $\|\v(t)-\ox\|\le\varepsilon$ and $\|\w(t)-\ooy(t)\|\le{\varepsilon(t)/\epsilon}$.
\item[\bf(c2)]$\|\v^*(t)\|\le 2\varepsilon(t)/\varepsilon$ and $\|\w^\ast(t)\|\leq 2\epsilon$.
\item[\bf(d2)]$\big|\varphi_t\big(\v(t),\w(t)\big)-\varphi_t\big(\ox,\ooy(t)\big)\big|\le\varepsilon(t)$.
\end{enumerate}\vspace*{0.05in}
Indeed, recalling that $\varphi_t$ is nonnegative by \eqref{lower_bound_assump2} tells us that the point  $(\ox,\ooy(t))$ is an $\varepsilon(t)$-minimizer of the function $\varphi_t$. Denoting now $\lambda_1(t):=\epsilon$ and $\lambda_2(t):=\epsilon(t)/\epsilon$, we deduce all the claim statements (a2)--(d2) from Lemma~\ref{lemmasubdeltaminimun}.\\[1ex]
{\bf Claim~3:} {\em Consider the measurable functions
\begin{equation*}
\begin{aligned}
\tilde\x(t)&:={\x}(t)\1_{A^c}(t)+\v(t)\1_{A}, & & & \tilde\x^*(t)&:={\x}^\ast(t)\1_{A^c}(t)+\v^\ast(t)\1_{A},\\
\tilde\y(t)&:={\y}(t)\1_{A^c}(t)+\w(t)\1_{A}, & & & \tilde\y^*(t)&:={\y}^\ast(t)\1_{A^c}(t)+\w^\ast(t)\1_{A}.
\end{aligned}
\end{equation*}
Then $\tilde\x\in \Leb^\infty({T},\X)$, $\tilde\x^\ast\in{\Leb}^1({T},\X)$, $\tilde\y\in\Leb^1(T,\Y)$, $\tilde\y^\ast\in\Leb^\infty(T,\Y)$, and}
\begin{enumerate}[leftmargin=0.4in,nolistsep,label=3.\arabic*),ref=3.\arabic*)]
\item[\bf(a3)] $\big(\tilde\x^*(t),\tilde\y^\ast(t)\big)\in\Hat\partial\varphi_t\big(\tilde\x(t),\tilde\y(t)\big)$ for a.e.\ $t\in T$.
\item[\bf(b3)] $\|\ox-\tilde\x\|_\infty\le\epsilon$ and $\|\ooy-\tilde\y\|_1\le 2\epsilon$.
\item[\bf(c3)] $\disp\int_T\big|\varphi_t\big(\tilde \x(t),\tilde \y(t)\big)-\varphi_t\big(\ox,\ooy(t)\big)\big|\dmu\le\epsilon+\epsilon^2$.
\item[\bf(d3)] $\Big\|\disp\int_T\tilde\x^*(t)\dmu-\ox^\ast\Big\|\le 5\epsilon$ and $\disp\int_T\|\tilde \x^*(t)\|\cdot\|\tilde\x(t)-x\|\dmu\le 3\epsilon^2$.
\item[\bf(e3)] $ \|{\tilde\y}^\ast\|_\infty \le 2\epsilon$.
\end{enumerate}\vspace*{0.03in}
Indeed, it follows directly from the above constructions that $(\tilde\x^*(t),\tilde \y^\ast(t))\in\Hat{\partial}\varphi_t(\tilde\x(t),\tilde\y(t))$ for a.e.\ $t\in T$, $\|x-\tilde\x\|_{\infty}\le\varepsilon$, $\|\ooy-\tilde\y\|_1\le 2\epsilon$, and
\begin{align*}
\disp\int_{T}\big|\varphi_t\big(\tilde\x(t),\tilde\y(t)\big)-\varphi_t\big(\ox,\ooy(t)\big)\big|\dmu\le\disp\int_{A^c}\big|\varphi_t\big(\x(t),\y(t)\big)-\varphi_t\big(\ox,\ooy(t)\big)\big|\dmu\\
+\disp\int_{A}\big|\varphi_t\big(\v(t),\w(t)\big)-\varphi_t\big(\ox,\ooy(t)\big)\big|\dmu\le\epsilon+\epsilon^2,
\end{align*}
\begin{equation*}
\disp\int_T\|\tilde\x^*(t)\|\dmu=\disp\int_{A^c}\|{\x}^*(t)\|\dmu+\disp\int_{A}\|{\v}^*(t)\|\dmu\le\mu({T})^{1/2}\|{\x}^*\|_2 +2\epsilon.
\end{equation*}
This readily leads us to the following inequalities:
\begin{align*}
\left\|\disp\int_T\tilde\x^*(t)\dmu-x^*\right\|&\le\left\|\int_{T}{\x}^*(t)\dmu-x^*\right\|+\int_{A}\|{\v}^*(t)\|\dmu\\
&+\disp\int_{A}\|{\x}^*(t)\|\dmu\le\varepsilon+2\epsilon+2\epsilon=5\epsilon,
\end{align*}
\begin{align*}
\disp\int_T\|\tilde\x^*(t)\|\cdot\|{x}-\tilde\x(t)\|\dmu\le&\disp\int_{A^c}\|{\x}^*(t)\|\cdot\|{x}-{\x}(t)\|\dmu\\&+\disp\int_{A}\|{\v}^*(t)\|\cdot\|{x}-{\v}(t)\|\dmu\\
\le&\epsilon^2+\disp\int_{A}\|{\v}^*(t)\|\left(\|\ox-{x}\|+\|\ox-{\v}(t)\|\right)\dmu\\
\le&\epsilon^2+\disp\int_{A}\big(\epsilon(t)/\epsilon\big)\cdot\left(\epsilon+\epsilon \right)\dmu\le\epsilon^2+2\epsilon^2=3\epsilon^2.
\end{align*}
Finally, we arrive at the norm estimates $\|{\tilde\y}^\ast\|_\infty\le\max\{\|\y^\ast\|_\infty,\|\w^\ast\|_\infty\}\le 2\epsilon$, which justify condition (a3)--(e3) and thus end the verification of this claim.\\[2ex]
{\bf Claim~4:} {\em Completing the proof of the theorem.} To finalize the proof of assertions (i)--(iv), we unify the results of the above claims to construct the desired sequences therein along an arbitrarily sequence $\ve_k\dn 0$ as $k\to\infty$.
\end{proof}

To conclude the paper, we present the following consequence of Theorem~\ref{theorem_main_fuzzy_sub} that provides a sequential calculus rule for {\em every} regular subgradient of the extended-integral functional under a certain local growth condition.

\begin{corollary}[\bf stronger convergence under another growth condition]\label{str-conv} Consider the expected-integral functional $\Intf{\varphi}$ generated in \eqref{eif} by a normal integrand $\varphi\colon\T\times\X\times\Y\to\Rex$, and let $(\ox,\ooy)$ be a point satisfying assumption \eqref{convex_cond}. In addition, suppose that there exist $\hat\nu,\hat\kappa\in\Leb^1(T,\R_+)$, and $\hat M,\hat r>0$ such that for all $(t,v,w)\in T\times\mathbb{B}_{\hat r}(\ox)\times\Y$ we have
\begin{equation*}
\varphi_t(v, w)\ge-\hat M\|w\|-\hat\nu(t),\text{ and }\;\dom\varphi_t(v,\cdot)\subset\hat\kappa(t)\mathbb{B},
\end{equation*}
Then for all $(\ox^*,\ooy^\ast)\in\Hat{\partial}\Intf{\varphi}(\ox,\ooy)$ there exist sequences $\{x_k\}\subset\X$, $\{\x_k\}\subset\Leb^\infty({T},\X)$, $\{\x_k^\ast\}\subset{\Leb}^1({T},\X)$, $\{\y_k\}\subset\Leb^1(T,\Y)$, $\{\y_k^\ast\}\subset\Leb^\infty(T,\Y)$ with:
\begin{enumerate}[label=\alph*),ref=\alph*)]
\item[\bf(i)] $\big(\x_k^*(t),\y_k^\ast(t)\big)\in\Hat\partial\varphi_t\big(\x_k(t),\y_k(t)\big)$ for a.e.\ $t\in T$ and all $k\in\N$.
\item[\bf(ii)] $\|\ox-x_k\|\to 0$, $\|\ox-\x_k\|_\infty\to 0$,   $\|\ooy-\y_k\|_1\to 0$, $\|\y_k^\ast-\ooy \|_\infty\to 0$ as $k\to\infty$.
\item[\bf(iii)] $\disp\Big\|\int_T\x_k^*(t)\dmu-\ox^\ast\Big\|\to 0$,  $\disp\int_T\|\x_k^*(t)\|\cdot\|\x_k(t)-x_k\|\dmu\to 0$ as $k\to\infty$.
\item[\bf(iv)] $\disp\int_T\big|\varphi_t\big(\x_k(t),\y_k(t)\big)-\varphi_t\big(\ox,\ooy(t)\big)\big|\dmu\to 0$ as $k\to\infty$
\end{enumerate}
\end{corollary}
\begin{proof}
Using the imposed growth condition, we have that the estimate
\begin{align*}
\inf\limits_{v\in\mathbb{B}_{\hat r}(\ox),\;\w\in\Y}\left\{\varphi_t(v,w)-\langle\ooy^\ast(t),w-\ooy(t)\rangle\right\}\ge&-\nu(t),\quad t\in T,
\end{align*}
where $ \nu(t):=\hat M\hat{\kappa}(t)+\hat{\nu}(t)+\|\ooy^\ast\|_\infty\hat{\kappa}(t)+\|\ooy^\ast\|_\infty\|\ooy\|_1$ is an integrable function. We see that $\varphi$ satisfies both assumptions   \eqref{convex_cond} and \eqref{growth_2}. Applying then Theorem~\ref{theorem_main_fuzzy_sub} gives us all the assertions (i)--(iv) of the corollary.
\end{proof}\\
{\bf Acknowledgments}. The authors are grateful to anonymous referees for their helpful remarks that allowed us to improve the original presentation.
\end{document}